\documentclass[reqno,10pt]{amsart}
\usepackage{amsmath, amsfonts,amsthm,amssymb,amsbsy,upref,color,graphicx,amscd,hyperref,enumerate}
\usepackage[active]{srcltx}
\usepackage[latin1]{inputenc}
\usepackage{bbm}
\usepackage{mathrsfs}

\def\ds{\displaystyle}
\def\eps{{\varepsilon}}
\def\N{\mathbb{N}}
\def\Om{\Omega}

\def\la{\lambda}
\def\R{\mathbb{R}}

\def\F{\mathcal{F}}
\def\HH{\mathcal{H}}
\def\Dr{D}

\def\div{\mathrm{div}}

\def\qm{K}

\numberwithin{equation}{section}
\theoremstyle{plain}
\newtheorem{teo}{Theorem}[section]
\newtheorem{lemma}[teo]{Lemma}
\newtheorem{cor}[teo]{Corollary}
\newtheorem{prop}[teo]{Proposition}
\newtheorem{deff}[teo]{Definition}
\newtheorem{oss}[teo]{Remark}

\newcommand{\mean}[1]{\,-\hskip-1.08em\int_{#1}} 
\newcommand{\ind}{\mathbbm{1}}

\title{Regularity of the optimal sets for some spectral functionals}

\author{Dario Mazzoleni, Susanna Terracini, Bozhidar Velichkov}
\address{Dario Mazzoleni and Susanna Terracini: \newline \indent
 Dipartimento di Matematica ``Giuseppe Peano'', \newline \indent
Universit\`a di Torino, \newline \indent
Via Carlo Alberto, 10,
10123 Torino, Italy,} 
\email{dmazzole@unito.it, 
susanna.terracini@unito.it}

\address {Bozhidar Velichkov: \newline \indent
Laboratoire Jean Kuntzmann (LJK), Universit\'e Grenoble Alpes
\newline \indent
B\^atiment IMAG, 700 Avenue Centrale, 38401 Saint-Martin-d'H\`eres
}
\email{bozhidar.velichkov@imag.fr}

\date{\today}

\begin{document}

\begin{abstract}
In this paper we study the regularity of the optimal sets for the shape optimization problem \[
\min\Big\{\lambda_1(\Omega)+\dots+\lambda_k(\Omega)\ :\ \Omega\subset\R^d\ \text{open}\ ,\ |\Omega|=1\Big\},
\]	
where $\la_1(\cdot),\dots,\la_k(\cdot)$ denote the eigenvalues of the Dirichlet Laplacian and $|\cdot|$ the $d$-dimensional Lebesgue measure.
We prove that the topological boundary of a minimizer $\Om_k^*$ is composed of a relatively open \emph{regular part} which is locally a graph of a $C^{\infty}$ function and a closed \emph{singular part}, which is empty if $d<d^*$, contains at most a finite number of isolated points if $d=d^*$ and has Hausdorff dimension smaller than $(d-d^*)$ if $d>d^*$, where the natural number $d^*\in[5,7]$ is the smallest dimension at which minimizing one-phase free boundaries admit singularities.
	
To achieve our goal, as an auxiliary result, we shall extend for the first time the known regularity theory for the one-phase free boundary problem to the vector-valued case.
\end{abstract}

\keywords{Shape optimization, Dirichlet eigenvalues, optimality conditions, regularity of free boundaries, viscosity solutions}
\subjclass{49Q10 (35R35, 47A75, 49R05)}

\thanks{{\bf Acknowledgments.} 
D.~Mazzoleni and S.~Terracini are partially  supported ERC Advanced Grant 2013 n. 339958
{\it Complex Patterns for Strongly Interacting Dynamical Systems - COMPAT}, by the PRIN-2012-74FYK7 Grant {\it Variational and perturbative aspects of nonlinear differential problems}. B. Velichkov was partially supported by the project AGIR 2015 {\it M\'ethodes variationnelles en optimisation de formes-VARIFORM}\\
We thank Dorin Bucur  and Guido De Philippis for some useful discussions on the topic of this paper. In particular, Dorin Bucur showed us the importance of the Weiss monotonicity formula and the radial extension from Lemma~\ref{lemmadorin}. Guido De Philippis pointed the paper~\cite{dss} out to us.}

\maketitle

\tableofcontents
\section{Introduction}
Functionals involving the eigenvalues of the Laplacian are the object of a growing interest in the analysis of PDEs from Mathematical Physics. Particularly challenging are the links between the spectrum of the Laplace operator and the geometry of the domain, a typical example being the Weyl asymptotic law. In this paper we study the regularity properties of the sets $\Omega$ that minimize the sum $\lambda_1(\Omega)+\dots+\lambda_k(\Omega)$ of the first $k$ eigenvalues of the Dirichlet Laplacian among all sets of fixed volume. That is, we are interested in the solutions of the shape optimization problem 
\begin{equation}\label{introP}
\min\Big\{\lambda_1(\Omega)+\dots+\lambda_k(\Omega)\ :\ \Omega\subset\R^d\ \text{open}\ ,\ |\Omega|=1\Big\},
\end{equation}
where $\lambda_1(\Omega)\le\dots\le\lambda_i(\Omega)\le\dots\le \lambda_k(\Omega)$, for $i=1,\dots,k$, denote the eigenvalues of the Dirichlet Laplacian on the set $\Omega$ counted with the due multiplicity\footnote{We recall that on an open set of finite volume the Dirichlet Laplacian has compact resolvent and its spectrum is real and discrete.}. 

From the point of view of the shape optimization theory,  problem \eqref{introP} is a special model case of the more general spectral optimization problem \begin{equation}\label{intro1}
\min\Big\{F\big(\lambda_1(\Omega),\dots,\lambda_k(\Omega)\big)\,:\,\Omega\subset\R^d,\ |\Omega|=1\Big\},
\end{equation}
where the cost function is defined through a function $F:\R^k\to\R$. The optimization problems of the form \eqref{intro1} naturally arise in the study of physical phenomena as, for example, heat diffusion or wave propagation inside a domain $\Omega\subset\R^d$, for a detailed introduction to the topic we refer to the books \cite{bubu05,hp,h}. The solution of \eqref{intro1} is known explicitly only in the special cases $F(\lambda_1,\dots,\lambda_k)=\lambda_1$ and $F(\lambda_1,\dots,\lambda_k)=\lambda_2$. For more general functionals the existence of a solution in the class of quasi-open sets\footnote{A quasi-open set is a level set $\{u>0\}$ of a Sobolev function $u\in H^1(\R^d)$. In particular, every open set is also quasi-open.} was first proved by Buttazzo and Dal Maso in \cite{budm93} for $F$ increasing in each variable and lower semi-continuous, under the assumption that the candidate sets $\Omega$ are all contained in a bounded open set $D\subset \R^d$. This last assumption was later removed by Bucur in \cite{bulbk} and Mazzoleni and Pratelli in \cite{mp}.

The regularity of the optimal sets and of the corresponding eigenfunctions turns out to be a rather difficult issue, due to the min-max nature of the spectral cost functionals, and was an open problem since the general Buttazzo-Dal Maso existence theorem. The only known result prior to the present paper concerning the regularity of the free boundary of the optimal sets is due to Brian\c{c}on and Lamboley \cite{brla} who prove that the optimal sets for the problem 
\begin{equation}\label{lb1}
\min\big\{\lambda_1(\Omega)\,:\,\Omega\subset\Dr\ \text{open},\ |\Omega|=1\big\},
\end{equation}
in a bounded open set $D\subset\R^d$ have smooth boundary up to a set of finite $(d-1)$-dimensional Hausdorff measure. Based on the techniques introduced in the seminal paper of Alt and Caffarelli \cite{altcaf}, this result depends strongly on the fact that the first eigenvalue is the minimum of the variational problem
$$\lambda_1(\Omega)=\min\Big\{\int_{\R^d}|\nabla u|^2\,dx\,:\,u\in H^1_0(\Omega),\ \int_{\R^d} u^2\,dx=1\Big\}$$
and so the shape optimization problem \eqref{lb1} can be written as a one-phase free boundary problem
$$\min\Big\{\int_{D}|\nabla u|^2\,dx+\Lambda |\{u>0\}|\,:\,u\in H^1_0(\Dr),\ \int_{D} u^2\,dx=1\Big\},$$
where the level set $\{u>0\}$ corresponds to $\Omega$ and $\Lambda$ is a Lagrange multiplier. The extension of this result to the general case of functionals involving higher eigenvalues presents some major difficulties since the higher eigenvalues are variationally characterized through a min-max procedure and thus it is not possible to reduce the shape optimization problem \eqref{intro1} to a one-phase free boundary problem. 
Nevertheless, some properties of the optimal sets were deduced in \cite{bulbk}, \cite{mp}, \cite{bucma} and \cite{bmpv}, as for example the fact that they are bounded, have finite perimeter and Lipschitz continuous eigenfunctions. We summarize the known results for the functional $F\big(\lambda_1(\Omega),\dots,\lambda_k(\Omega)\big)=\lambda_1(\Omega)+\dots+\lambda_k(\Omega)$ in the following theorem.

\begin{teo}\label{knownstuffmainopt0}
\begin{enumerate}[(i)]
\item (Buttazzo-Dal Maso \cite{budm93}) Given a bounded open set $D\subset\R^d$, there is a solution to the shape optimization problem
\begin{equation*}
\min\Big\{\lambda_1(\Omega)+\dots+\lambda_k(\Omega)\ :\ \Omega\subset D\ \text{quasi-open},\ |\Omega|=1 \Big\}.
\end{equation*}
\item (Bucur \cite{bulbk}; Mazzoleni-Pratelli \cite{mp}) There is a solution to the shape optimization problem
\begin{equation}\label{sopBMP}
\min\Big\{\lambda_1(\Omega)+\dots+\lambda_k(\Omega)\ :\ \Omega\subset \R^d\ \text{quasi-open},\ |\Omega|=1 \Big\}.
\end{equation}
Moreover every solution $\Omega^\ast$ of \eqref{sopBMP} is bounded. 
\item (Bucur \cite{bulbk}) Every solution $\Omega^\ast$ of \eqref{sopBMP} has finite perimeter. 
\item (Bucur-Mazzoleni-Pratelli-Velichkov \cite{bmpv}) Let $\Omega^\ast$ be a solution of \eqref{sopBMP}. Then the first $k$ normalized eigenfunctions $u_1,\dots,u_k$ on $\Omega^\ast$, extended by zero over $\R^d\setminus \Omega^\ast$, are Lipschitz continuous on $\R^d$ and $\|\nabla u_i\|_{L^{\infty}}\le C_{d,k}$, for every $i=1,\dots,k$, where $C_{d,k}$ is a constant depending only on $k$ and $d$. In particular, every solution of \eqref{sopBMP} is an open set and is also a solution of \eqref{introP}.
\end{enumerate}
\end{teo} 

The aim of this paper is to prove that the boundary of the optimal sets, solutions of \eqref{introP}, is regular up to a set of lower dimension, precisely we prove that $\Omega^\ast$ is $d^\ast$-regular in the sense of the following definition.
\begin{deff}\label{d*regular}
We call a set $\Om\subset \R^d$ \emph{$d^*$-regular} if $\partial \Om$ is the disjoint union of a regular part $Reg(\partial \Om)$ and a (possibly empty), singular part $Sing(\partial\Omega)$ such that: 
\begin{itemize}
\item $Reg(\partial\Omega)$ is an open subset of $\partial \Omega$ and locally a $C^{\infty}$ hypersurface of codimension one;
\item $Sing(\partial \Om)$ is a closed subset of $\partial\Omega$ and has the following properties:
\begin{itemize}
\item If $d<d^*$, then $Sing(\partial \Om)$ is empty,
\item If $d=d^*$, then the singular set $Sing(\partial \Om)$ contains at most a finite number of isolated points,
\item If $d>d^*$, then the Hausdorff dimension of $Sing(\partial \Om)$ is less than $d-d^*.$
\end{itemize} 
\end{itemize}
\end{deff}
In our work, $d^*$ is the smallest dimension at which the 
free boundaries of the local minima of scalar the one-phase functional
$$u\mapsto \int|\nabla u|^2\,dx+|\{u>0\}|,$$
admit singularities. Up to our knowledge $d^*\in[5,7]$, see~\cite{dsj} and the recent work~\cite{js}.
The main result of the paper is the following.
\begin{teo}\label{main}
Let the open set $\Om^*_k\subset\R^d$ be an optimal set for problem~\eqref{introP}. Then $\Om^*_k$ is connected and $d^*$-regular. 
Moreover the vector $U=(u_1,\dots, u_k)$ of the normalized eigenfunctions is such that $|U|$ has a $C^1$ extension on the regular part of the free boundary and satisfies the optimality condition
\begin{equation}\label{optcondintro}
\big|\nabla |U|\big|=\sqrt{\Lambda}\quad\text{on}\quad Reg(\Om^*_k),
\end{equation}
where the constant $\Lambda$ is given by $\ds\Lambda=\frac{2}{d}\sum_{i=1}^k\lambda_i(\Omega_k^\ast)$.
\end{teo}
{\begin{proof}[Proof of Theorem~\ref{main}]
The fact that $\Om^*_k$ is connected will be proved in Corollary \ref{mainconnected}. The regular part of the free boundary will be the object of Proposition \ref{regularityfinal} and of Proposition~\ref{regularityfinalinfty}, while for the singular part we refer to Proposition \ref{regbdryweiss}. The extremality condition~\eqref{optcondintro} is a consequence of the optimality condition in viscosity sense (see Lemma \ref{optvisc}) and the fact that $\nabla |U|$ is well defined on the regular part of the free boundary.	
\end{proof}}

In order to prove Theorem \ref{main} we first show that the vector of eigenfunctions $U=(u_1,\dots,u_k)$ is a local quasi-minimizer of the vector-valued functional 
\begin{equation*}
H^1(\R^d;\R^k)\ni V\mapsto \int_{\R^d}|\nabla V|^2\,dx+\Lambda\big|\{|V|>0\}\big|,
\end{equation*}
that is, $U$ is a local minimizer of the functional
$$H^1(\R^d;\R^k)\ni V\mapsto \Big(1+\qm\|V-U\|_{L^1}\Big)\int_{\R^d}|\nabla V|^2\,dx+\Lambda\big|\{|V|>0\}\big|.$$
Our proofs mostly rely on the free boundary approach for this shape optimization problem, suitably modifying many seminal ideas from~\cite{rtt, altcaf, weiss99},  that we are extending for the first time to the vectorial case.
The intrinsic differences are mainly related with the vectorial nature of the variable $U$.  This causes a number of new difficulties, starting from the non-degeneracy at the boundary, the classification of conic blow-ups, the validity and consequences of the extremality condition in a proper sense.  
We first use a Weiss-like monotonicity formula to classify the boundary points through a blow-up analysis.
Then, a key point of our argument is to prove an optimality condition~\eqref{optcondintro} for $|U|$ on the boundary, which is fulfilled in a proper viscosity sense. In the scalar case  this is a well-established approach, for which classical references are \cite{Caffa1,Caffa2}, which however cannot be easily reproduced in the vectorial case.
Next, in order to reduce our problem to a scalar one, we need to compare the boundary derivatives of the different components involved in the optimality condition.
We first prove that the regular part of the free boundary is Reifenberg flat, which implies that it is an NTA domain, following the works by Kenig and Toro~\cite{kt,kt1}.
For NTA domains, Jerison and Kenig~\cite{jk} proved a boundary Harnack inequality, which is enough for our aims.
Then we are able to obtain an optimality condition which involves only $u_1$ on the regular part of the free boundary and then apply the classical results to obtain $C^{1,\alpha}$ regularity.
In order to get $C^\infty$ regularity with a bootstrap argument, we need an improved boundary Harnack principle~\cite{dss}, which allows us to use the general result by Kinderlehrer and Nirenberg~\cite{kn} on the one-phase problem for $u_1$, which otherwise would not work directly in the vectorial setting.  
Finally, the analysis of the dimension for the singular set follows as in~\cite[Section 4]{weiss99} by an adaptation of the classical arguments from the theory of minimal surfaces. \\

\noindent{\bf Further remarks and comments.} As a consequence of the regularity theory developed for vector-valued functions, we obtain an auxiliary regularity result, which better highlights the analogy with the free boundary problem studied by Alt and Caffarelli~\cite{altcaf} and Weiss~\cite{weiss99}. We note that the extension to the vectorial case that we are able to prove still requires one function to have a positive trace (and so to be positive in the interior). A major open problem, up to our knowledge, is to prove Theorem~\ref{altcafvect} with all the $\phi_i$ changing sign on $\partial D$. How to deduce Theorem~\ref{altcafvect} from our arguments is explained in Section~\ref{sec:thm1.4}.

\begin{teo}\label{altcafvect}
Let $D\subset\R^d$ be an open set with smooth boundary, $\Lambda>0$, and let $\phi_1,\dots,\phi_k\in C^0(\partial D)$ be given functions, with $\phi_1>0$ on $\partial D$. Then, there is a solution $U=(u_1,\dots,u_k)\in H^1(D;\R^k)$ to the problem
\begin{equation}\label{eqaltcafvect}
\begin{split}
\min \Big\{\int_{D}{|\nabla U|^2\,dx}+\Lambda\left|\{|U|>0\}\right|,\;U\in H^1(D;\R^k),\ u_i=\phi_i\;\mbox{on } \partial D,\ \forall i=1,\dots,k\Big\}.
\end{split}
\end{equation}
Moreover, for every solution $U=(u_1,\dots, u_k)$ the set $\{|U|>0\}$ is $d^*$-regular and the optimality condition~\eqref{optcondintro} holds on the regular part of the free boundary.
\end{teo}

\begin{oss}
We highlight that in Theorem~\ref{altcafvect} above, the hypothesis $\phi_1>0$  is not the optimal one. In fact it is sufficient to suppose that, in each connected component of the open set $\{|U|>0\}$, there is at least one component $u_i$ of the vector $U$ which is positive. This holds for example if all $\phi_i$ are non-negative (as it is required in~\cite{csy}).
\end{oss}

Our results can be extended to the case of smooth functionals $F(\lambda_1,\dots,\lambda_k)$ which are invariant under permutations of the variables and non-decreasing in each variable.  The sum of powers of the first $k$ eigenvalues for example is of great interest also from the point of view of applications to the Lieb--Thirring theory, as it is explained by Lieb and Loss in~\cite[Chapter 12]{ll}, and it can be considered a more natural functional to study than the lone $\la_k$, when one has in mind, for example, the Lieb--Thirring inequalities. An extension of Theorem~\ref{main} to more general functionals of eigenvalues of the form~\eqref{intro1} (still involving $\la_1$) can be proved starting from the techniques of this work with some careful approximation procedures and will be the object of a forthcoming paper.

An alternative approach to the regularity of its solutions would be to see \eqref{introP} as a two-partition problem of $\R^d$ with the Lebesgue measure being the cost functional for one of the two competing populations and the sum of the eigenvalues the cost functional for the other one. Indeed, functionals involving higher eigenvalues were successfully treated in the framework of the  \emph{optimal partition problems}, for example in the recent work~\cite{rtt} (see also~\cite{tt}), where it is proved the existence of an optimal \emph{regular} partition, i.e. with free boundary that is $C^{1,\alpha}$ regular, up to a set of Hausdorff dimension less than $d-2$. Unfortunately, some key techniques used for partitions fail when dealing with~\eqref{introP}. For example, we are not able to establish an Almgren monotonicity formula, which is one of the principal tools used in \cite{rtt}. This is due, mainly, to the measure term, which does not seem to behave well with the quantities involved in the Almgren quotient.

As it was proved in \cite{bulbk} an optimal set $\Omega_k^\ast$ for \eqref{introP} has finite perimeter $P(\Omega_k^\ast)<\infty$. This means that there is a constant $P>0$ such that $\Omega_k^\ast$ is also a solution to the problem
\begin{equation*}
\min\Big\{\lambda_1(\Omega)+\dots+\lambda_k(\Omega)\ :\ \Omega\subset\R^d\ ,\ |\Omega|=1,\ P(\Omega)=P\Big\}.
\end{equation*}
Unfortunately, up to our knowledge, there is no way to directly replace the condition $P(\Omega)=P$ by a (non-zero) Lagrange multiplier or to reasonably approximate $\Omega_k^\ast$ by optimal sets for the functional $\lambda_1(\Omega)+\dots+\lambda_k(\Omega)+\Lambda P(\Omega)$, for which a regularity theory was developed in \cite{deve} (see also~\cite{bucma}).

\begin{oss}
The study of the optimal sets for the problem \eqref{introP} might suggest a new approach to some inequalities involving the spectrum of the Dirichlet Laplacian, as the well-known Li-Yau inequality \cite{liyau}, or to more refined lower bounds on $\lambda_1(\Omega)+\dots+\lambda_k(\Omega)$ in terms of the geometry of $\Omega$, as for example the ones suggested by the Weyl's asymptotic expansion. 
\end{oss}

{\noindent{\bf Plan of the paper.} In Section~\ref{sec:nondeg} we deal with the quasi-minimality of the eigenfunctions for a more general free boundary problem and then we provide some non-degeneracy and density estimates.
In Section~\ref{ssmono} we prove a monotonicity formula in the spirit of Weiss \cite{weiss99}.
In Section~\ref{sec:blowup} we perform the analysis of the blow-up limits and prove their optimality and $1$-homogeneity.
Finally, in Section~\ref{sec:regularity} we are ready to prove the regularity of the free boundary. We study the optimality condition in the viscosity sense, we identify the regular and singular part of the topological boundary and then we reduce ourselves to a problem with only one non-negative function and apply the regularity result for the classical Alt-Caffarelli free boundary problem.
At the end we provide the estimates on the Hausdorff dimension for the singular part of the boundary.
Section~\ref{sec:thm1.4} is devoted to highlight how with a similar scheme also Theorem~\ref{altcafvect} can be proved.}\\

\noindent {\bf Note.} 
After the submission and the upload on arXiv of this paper, we discovered the preprint~\cite{csy} by Caffarelli-Shahgholian-Yeressian, which appeared few days before ours. Our Theorem~\ref{altcafvect} is very similar to their main result, which requires the additional hypothesis that all $\phi_i$ are non-negative.
We stress that the two teams agreed that they worked in a completely independent way.

A recent preprint~\cite{kl} by Kriventsov-Lin appeared on arXiv few days later than ours. It contains a result similar to our Theorem~\ref{main}, for a slightly more general class of functionals. We point out that our result is stronger: whereas we prove $C^{\infty}$ regularity of the free boundary, up to a $d-5$ dimensional set, they prove only $C^{1,\alpha}$ regularity up to a $d-3$ dimensional set, with completely different techniques.

\medskip
\medskip
\noindent {\bf Preliminaries and notations.} We will denote by $d$ the dimension of the space and by $C_d$ a generic constant depending only on the dimension. For $x=(x_1,\dots,x_d)\in \R^d$ and $r>0$ we will denote by $B_r(x)$ the ball centered in $x$ of radius $r$ with respect to the Euclidean distance $|y|=(y_1^2+\dots+y_d^2)^{1/2}$. We will use the notation $B_r$, when the ball is centered in zero. For a generic measurable set $\Omega\subset\R^d$, by $|\Omega|$ we denote the Lebesgue measure of $\Omega$, while for the measure of the unit ball $B_1\subset\R^d$ we will use the notation $\omega_d$. 
For a point $x_0\in\R^d$ we recall that the density of the measurable set $\Omega$ in $x_0$ is given by 
$$\lim_{r\to0}\frac{|\Omega\cap B_r(x_0)|}{|B_r|},$$
whenever the above limit exists. We recall the classical notation
\begin{equation*}
\Omega^{(\gamma)}:=\Big\{x_0\in\R^d\ :\ \lim_{r\rightarrow 0}{\frac{|\Omega\cap B_r(x_0)|}{|B_r|}}=\gamma\Big\},
\end{equation*}
for the set of point of density $\gamma\in[0,1]$.
For $\alpha>0$ we will denote by $\HH^\alpha$ the $\alpha$-dimensional Hausdorff measure, for example the surface area of the unit sphere is $\HH^{d-1}(\partial B_1)=d\omega_d$. By $d_{\HH}(A,B)$ we denote the Hausdorff distance between the sets $A,B\subset\R^d$, 
\begin{equation*}
d_{\HH}(A,B):=\max{\left\{\sup_{a\in A}{\{\text{dist}(a,B)\}};\sup_{b\in B}{\{\text{dist}(b,A)\}}\right\}},
\end{equation*}
where for $x\in\R^d$ and $A\subset\R^d$ we set $\text{dist}(x,A)=\inf_{y\in A}|x-y|$.

For an open set $\Omega\in \R^d$ we denote with $H^1_0(\Omega)$ the Sobolev space obtained as a closure of the smooth real-valued functions with compact support $C^\infty_c(\Omega)$ with respect to the Sobolev norm $\ds\|u\|_{H^1}=\left(\int_\Omega |\nabla u|^2\,dx+\int_\Omega u^2\,dx\right)^{1/2}$. For a vector valued function $U=(u_1,\dots,u_k):\Omega\to\R^k$ we will say that $U\in H^1_0(\Omega;\R^k)$ if all of its components are Sobolev, $u_i\in H^1_0(\Omega)$ for every $i=1,\dots,k$,. Thus we have 
$$|U|^2=u_1^2+\dots+u_k^2\ ,\quad |\nabla U|^2=|\nabla u_1|^2+\dots+|\nabla u_k|^2\quad\text{and}\quad \|U\|_{H^1}=\left(\int_\Omega |\nabla U|^2\,dx+\int_\Omega |U|^2\,dx\right)^{1/2}.$$ 
If $\Omega=\R^d$, then the index zero will be omitted and we will use the usual notations $H^1(\R^d)$ and $H^1(\R^d;\R^k)$, for the vector-valued functions. Moreover, we will suppose that all the Sobolev functions $u\in H^1_0(\Omega)$ and $U\in H^1_0(\Omega;\R^k)$ are extended by zero outside $\Omega$. Thus $H^1_0(\Omega;\R^k)\subset H^1(\R^d;\R^k)$. 

Let $\Omega\subset\R^d$ be an open set of finite Lebesgue measure $|\Omega|<\infty$. The spectrum $\sigma(\Omega)$ of the Dirichlet Laplacian on $\Omega$ is given by an increasing sequence $\lambda_1(\Omega)\le \lambda_2(\Omega)\le\dots\le\lambda_k(\Omega)\le\dots$, of strictly positive, non-necessarily distinct real numbers. We call the elements of $\sigma(\Omega)$ eigenvalues and we count them with the due multiplicity. A real number $\lambda$ is an eigenvalue if there exists a non-trivial function $u\in H^1_0(\Omega)$ (an eigenfunction) solution of the equation 
$$-\Delta u=\lambda u\quad\text{in}\quad\Omega\ ,\qquad u\in H^1_0(\Omega)\ ,\qquad \int_\Omega u^2\,dx=1.$$
We will denote by $u_k$ the eigenfunction corresponding to the eigenvalue $\lambda_k(\Omega)$. The family of eigenfunctions $\{u_k\}_{k\in\N}$ form a (complete) orthonormal system in $L^2(\Omega)$, that is, 
$$\int_\Omega u_i u_j\,dx=\delta_{ij}:=\begin{cases}1,\ \text{if}\ i=j,\\ 0,\ \text{if}\ i\neq j.\end{cases}$$
The supremum of an eigenfunction on a set $\Omega$ can be estimated by a power of the corresponding eigenvalue independently on the regularity and the geometry of $\Omega$. The following estimate was proved in  \cite[Example 2.1.8]{davies}
$$\|u_k\|_{L^{\infty}(\R^d)}\le e^{1/8\pi}\lambda_k(\Omega)^{d/4}.$$
First of all we use capital letters for denoting vectors of functions like $U=(u_1,\dots, u_k)$ and we denote by $\Om_U:=\left\{x\in\R^d\;:\;|U(x)|>0\right\}$.

The eigenvalues of the Dirichlet Laplacian on $\Omega$ can be variationally characterized by the following min-max principle
$$\lambda_k(\Omega)=\inf_{S_k\subset H^1_0(\Omega)}\ \sup_{S_k\setminus \{0\}}\ \frac{\int_\Omega |\nabla u|^2\,dx}{\int_\Omega u^2\,dx},$$
where the infimum is over all $k$-dimensional linear subspaces $S_k$ of $H^1_0(\Omega)$. Thus, for $\lambda_1(\Omega)$ we have 
$$\lambda_1(\Omega)=\inf_{u\in H^1_0(\Omega)\setminus \{0\}}\ \frac{\int_\Omega |\nabla u|^2\,dx}{\int_\Omega u^2\,dx}.$$
A similar variational formulation, involving vector-valued functions, holds for the sum of the first $k$ eigenvalues (see for example \cite{ll} or \cite{rtt})
\begin{equation}\label{somme}
\sum_{i=1}^k\lambda_i(\Omega)=\min\Big\{\int_\Omega |\nabla U|^2\,dx\ :\ U=(u_1,\dots,u_k)\in H^1_0(\Omega;\R^k),\ \int_\Omega u_i u_j\,dx=\delta_{ij}\Big\},
\end{equation}
the minimum being attained for the vector $U$ whose components are the first $k$ normalized eigenfunctions on $\Omega$. 

Viewed as a a functional over the family of open sets, $\lambda_k(\cdot)$ is decreasing with respect to the set inclusion and is homogeneous of order $-2$, i.e. we have that for any $t>0$
\begin{equation}\label{homoeigen}
\lambda_k(t\Omega)=\frac{1}{t^2}\lambda_k(\Omega)\qquad\text{and}\qquad\sum_{i=1}^k\lambda_i(t\Omega)=\frac1{t^2}\sum_{i=1}^k\lambda_i(\Omega),\end{equation}
where, as usual, we denote by $t\Omega$ the set $\ds t\Omega:=\{x\in\R^d\ :\ \frac{x}{t}\in\Omega\}.$

\section{Properties of the eigenfunctions on the optimal sets}\label{sec:nondeg}

In this section we study the normalized eigenfunctions on an optimal set for problem~\eqref{introP}. We will denote by $\Omega$ a solution of \eqref{introP} and by $U$ the corresponding vector of normalized eigenfunctions on $\Omega$, $U=(u_1,\dots,u_k)$. We also set
$\ds\Lambda:=\frac2d\sum_{i=1}^k\lambda_i(\Omega).$ 

In subsection \ref{sec:eigen} we will show that $U$ is a local quasi-minimizer of a variational problem in the sense of the following proposition.
\begin{prop}[Minimality of $U$]\label{quasimin_brno}
Suppose that the set $\Omega \subset\R^d$ is a solution to the shape optimization problem \eqref{introP}. Then the vector $U=(u_1,\dots,u_k)\in H^1_0(\Omega;\R^k)$ of normalized eigenfunctions on $\Omega$ satisfies the following quasi-minimality condition:
\begin{equation}\label{Uoptcond}
\begin{split}
&\text{There are constants $\qm>0$ and $\eps>0$ such that}\\ 
&\int_{\R^d}|\nabla U|^2\,dx+\Lambda\big|\{|U|>0\}\big|\le \Big(1+\qm\|U-\widetilde U\|_{L^1}\Big)\int_{\R^d}|\nabla \widetilde U|^2\,dx+\Lambda\big| \{|\widetilde U|>0\}\big|,\\
&\text{for every}\quad \widetilde U\in H^1(\R^d;\R^k)\quad\text{such that}\quad \|\widetilde U\|_{L^\infty}\le \eps^{-1}\quad\text{and}\quad \|U-\widetilde U\|_{L^1}\le \eps. 
\end{split}
\end{equation}
\end{prop}

In subsection \ref{sec:nondeg2} we will use Proposition \ref{quasimin_brno} to show that the vector of the eigenfunctions on the optimal set does not degenerate at the free boundary. The following proposition describes the behavior of the eigenfunctions close to the boundary. We notice that the first claim is simply a restatement of Theorem \ref{knownstuffmainopt0} (iv).
\begin{prop}[Boundary behavior of the eigenfunctions]\label{prop:nondeg}
Let $\Omega$ be optimal for \eqref{introP} and let $U=(u_1,\dots,u_k)\in H^1_0(\Omega;\R^k)$ be the vector of the first $k$ normalized eigenfunctions on $\Omega$. 
\begin{enumerate}
\item The vector-valued function $U:\R^d\to\R^k$ is Lipschitz continuous on $\R^d$. 
\item The real-valued function $|U|$ is non-degenerate, i.e. there are constants $c_0>0$ and $r_0>0$ such that for every $x_0\in\R^d$ and $r\in(0,r_0]$ the following implication holds
\begin{equation*}
\Big(\mean{B_r(x_0)}{|U|\,dx}<c_0 r\Big)\Rightarrow \Big(U\equiv 0\ \ \text{in}\ \ B_{r/2}(x_0)\Big).
\end{equation*}
\item The first eigenfunction $u_1$ is non-degenerate, i.e. there are constants $c_0>0$ and $r_0>0$ such that for every $x_0\in\R^d$ and $r\in(0,r_0]$ the following implication holds
\begin{equation*}
\Big(\mean{B_r(x_0)}{u_1\,dx}<c_0 r\Big)\Rightarrow \Big(u_1\equiv 0\ \ \text{in}\ \ B_{r/2}(x_0)\Big).
\end{equation*}
\end{enumerate}
\end{prop}

As a corollary of Proposition \ref{prop:nondeg} we obtain that the optimal sets for \eqref{introP} satisfy a density estimate. 

\begin{cor}[Density estimate]\label{densestcor}
Let $\Omega$ be optimal for \eqref{introP}. Then $\Omega=\{|U|>0\}$ and there are constants $\eps_0$, $r_0$ and $\delta$ such that:
\begin{enumerate}
\item  The following density estimate holds: 
\begin{equation*}
\eps_0|B_r|\le \big|\Omega\cap B_r(x_0)\big|\le (1-\eps_0)|B_r|,\quad\text{for every}\quad x_0\in \partial\Omega\quad\text{and}\quad r\le r_0.   
\end{equation*}
\item For every $x_0\in\partial\Omega$ and $r\le r_0$ there is a point $x_1\in \partial B_{r/2}(x_0)$ such that $B_{\delta r}(x_1)\subset\Omega$.
\end{enumerate}
\end{cor}

\subsection{Quasi-minimality of the eigenfunctions}\label{sec:eigen}
In this subsection we prove that the vector of eigenfunctions $U\in H^1_0(\Omega;\R^k)$ on the optimal set $\Omega$ for \eqref{introP} is a local minimum of a functional of the form
$$\mathcal F_{\qm}:H^1(\R^d;\R^k)\to\R\ ,\qquad\mathcal{F}_{\qm}(V)=\Big(1+\qm\|U-V\|_{L^1}\Big)\int_{\R^d}|\nabla V|^2\,dx+\Lambda\big| \{|V|>0\}\big|,$$
that can alternatively be interpreted as a local quasi-minimum of the functional 
$$\mathcal{F}_0(V)=\int_{\R^d}|\nabla V|^2\,dx+\Lambda\big| \{|V|>0\}\big|.$$
We first prove the following Lemma which assures the existence of the Lagrange multiplier for \eqref{introP}.
\begin{lemma}\label{laglem}
Suppose that $\Omega$ is a solution of \eqref{introP}. Then $\Omega$ is a solution of the shape optimization problem
\begin{equation*}
\min\Big\{\lambda_1(\widetilde\Omega)+\dots+\lambda_k(\widetilde\Omega)+\Lambda|\widetilde\Omega|\ :\ \widetilde\Omega\subset\R^d\ \text{open}\Big\},
\end{equation*}
where $\ds\Lambda=\frac2d\sum_{i=1}^k\lambda_i(\Omega)$.
\end{lemma}
\begin{proof}
Let $\widetilde\Omega\subset\R^d$ be a generic open subset of $\R^d$ of finite Lebesgue measure. By the optimality of $\Omega$ and the homogeneity of the eigenvalues \eqref{homoeigen} we have that
$$\sum_{i=1}^k\lambda_i(\widetilde\Omega)\ge \sum_{i=1}^k\lambda_i(t\Omega) =\frac1{t^2}\sum_{i=1}^k\lambda_i(\Omega),$$
where $t$ is such that $|t\Omega|=t^d|\Omega|=|\widetilde\Omega|$. Thus, we have
\begin{align*}
\sum_{i=1}^k\lambda_i(\widetilde\Omega)+\Lambda|\tilde\Omega|&\ge \frac1{t^2}\sum_{i=1}^k\lambda_i(\Omega) + t^d\Lambda|\Omega|\ge \sum_{i=1}^k\lambda_i(\Omega) + \Lambda|\Omega|,
\end{align*}
where the last inequality is due to the fact that the function 
$$t\mapsto \frac1{t^2}\sum_{i=1}^k\lambda_i(\Omega) + t^d\Lambda|\Omega|,$$ achieves its maximum at $t=1$.
\end{proof}

In view of the variational characterization \eqref{somme} of the sum of the first $k$ eigenvalues and Lemma \ref{laglem} we have that $U$ is a solution of the problem 
\begin{equation}\label{sommeV}
\min\Big\{\int_{\R^d}|\nabla V|^2\,dx+\Lambda|\{|V|>0\}|\ :\ V=(v_1,\dots,v_k)\in H^1(\R^d;\R^k),\ \int_{\R^d}v_iv_j\,dx=\delta_{ij}\Big\}.
\end{equation}
In the following lemma we remove the orthogonality constraint $\ds\int_{\R^d}v_iv_j\,dx=\delta_{ij}$.
\begin{lemma}[Orthonormalization of small perturbations]\label{lemmaorto}
Let $U=(u_1,\dots,u_k)$, where $u_1,\dots,u_k$ are eigenfunctions on an open domain $\Omega$. Let $\delta>0$ be fixed, and let $\tilde U=(\tilde u_1,\dots,\tilde u_k)\in H^1(\R^d;\R^k)$ be such that 
$$\eps_k:=\sum_{i=1}^k \int_{B_{r}}|\tilde u_i-u_i|\,dx\le 1\qquad\text{and}\qquad \sup_{i=1,\dots,k}\Big\{\|u_i\|_{L^\infty(B_r)}+\|\tilde u_i\|_{L^\infty(B_r)}\Big\}\le \delta.$$
Let $V=(v_1,\dots, v_k)\in H^1_0(\Omega\cup B_r)$ be the vector obtained  orthonormalizing $\tilde U$ by the Gram-Schmidt procedure, i.e. 
\begin{equation*}
\begin{array}{lll}
v_1&=&\|\tilde u_1\|_{L^2}^{-1}\tilde u_1,\\
v_2&=&\Big\|\tilde u_2-\Big(\int \tilde u_2 v_1\,dx\Big)v_1\Big\|_{L^2}^{-1}\Big(\tilde u_2-\Big(\int \tilde u_2 v_1\,dx\Big)v_1\Big),\\
v_3&=&\Big\|\tilde u_3-\Big(\int \tilde u_3 v_2\,dx\Big)v_2-\Big(\int \tilde u_2 v_1\,dx\Big)v_1\Big\|_{L^2}^{-1}\Big(\tilde u_3-\Big(\int \tilde u_3 v_2\,dx\Big)v_2-\Big(\int \tilde u_2 v_1\,dx\Big)v_1\Big),\\
...\\
v_k&=&\Big\|\tilde u_k-\sum_{i=1}^{k-1}\Big(\int \tilde u_k v_i\,dx\Big)v_i\Big\|_{L^2}^{-1}\Big(\tilde u_k-\sum_{i=1}^{k-1}\Big(\int \tilde u_k v_i\,dx\Big)v_i\Big).
\end{array}
\end{equation*}

There exist constants $1\ge \overline\eps_k>0$ and $\overline C_k>0$, depending on the dimension $d$, the constant  $k$, the bound $\delta$ and the measure $|\Omega|$,  such that the following estimate holds for every $\tilde U$ as above with $\eps_k\le \overline \eps_k$.
\begin{equation}\label{maingradest08}
\int_{\R^d}|\nabla V|^2\,dx\le \Big(1+\overline C_k \eps_k\Big)\int_{\R^d}|\nabla \tilde U|^2\,dx.
\end{equation}
\end{lemma}
\begin{proof}
We first prove that there is $\overline\eps_k$ and $C_k$ such that the following estimates hold whenever $\eps_k\le \overline \eps_k$.
\begin{align*}
\sum_{i=1}^k\|u_i-v_i\|_{L^1}\le C_k\eps_k,\\
\max_{i=1,\dots, k}\|v_i\|_{L^\infty}\le C_k,
\end{align*}
where $C_{k}$ and $\overline\eps_k$ are constants depending on the dimension $d$, the constant  $k$, the bound $\delta$ and the measure $|\Omega|$.
We proceed by induction. In fact for $k=1$ we have 
\begin{equation}\label{firstindstepest08}
\begin{split}
\|u_1-v_1\|_{L^1}&\le\|u_1-\tilde u_1\|_{L^1}+\|\tilde u_1-v_1\|_{L^1}=\|u_1-\tilde u_1\|_{L^1}+\frac{\big|\|\tilde u_1\|_{L^2}-1\big|}{\|\tilde u_1\|_{L^2}}\|\tilde u_1\|_{L^1}\\
&\le\|u_1-\tilde u_1\|_{L^1}+\frac{\big|\|\tilde u_1\|_{L^2}^2-1\big|}{\|\tilde u_1\|_{L^2}^2}\|\tilde u_1\|_{L^1}\\
&=\|u_1-\tilde u_1\|_{L^1}+\frac{\big|\|u_1+(\tilde u_1-u_1)\|_{L^2}^2-1\big|}{\| u_1+(\tilde u_1-u_1)\|_{L^2}^2}\|u_1+(\tilde u_1-u_1)\|_{L^1}\\
&=\|u_1-\tilde u_1\|_{L^1}+\frac{2\int u_1|\tilde u_1-u_1|\,dx+\|\tilde u_1-u_1\|_{L^2}^2}{1-2\int u_1|\tilde u_1-u_1|\,dx}\Big(\|u_1\|_{L^1}+\|\tilde u_1-u_1\|_{L^1}\Big)\\
&=\frac{1+\|\tilde u_1-u_1\|_{L^2}^2}{1-2\int u_1|\tilde u_1-u_1|\,dx}\|u_1-\tilde u_1\|_{L^1}+\frac{2\int u_1|\tilde u_1-u_1|\,dx+\|\tilde u_1-u_1\|_{L^2}^2}{1-2\int u_1|\tilde u_1-u_1|\,dx}\|u_1\|_{L^1}\\
&\le\frac{1+\|\tilde u_1-u_1\|_{L^1}\|\tilde u_1-u_1\|_{L^\infty}+\|u_1\|_{L^1}\Big(2\|u_1\|_{L^\infty}+\|\tilde u_1-u_1\|_{L^\infty}\Big)}{1-2\|u_1\|_{L^\infty}\|\tilde u_1-u_1\|_{L^1}}\|u_1-\tilde u_1\|_{L^1}\\
&\le\frac{1+\delta\eps_1+|\Omega|^{1/2}4\delta}{1-2\delta\eps_1}\eps_1\le \big(1+12\delta |\Omega|^{1/2}\big) \eps_1,
\end{split}
\end{equation}
where the last inequality holds for $\eps_1\le\inf\Big\{\delta,(4\delta)^{-1},|\Omega|^{1/2}\Big\}.$ On the other hand, for the infinity norm we have 
\begin{equation}\label{firstindstepest09}
\begin{split}
\|v_1\|_{L^\infty}&=\frac{\|\tilde u_1\|_{L^\infty}}{\|\tilde u_1\|_{L^2}}=\frac{\|\tilde u_1\|_{L^\infty}}{\|u_1+(\tilde u_1-u_1)\|_{L^2}}\le\frac{\|\tilde u_1\|_{L^\infty}}{\left(1-2\int u_1|\tilde u_1-u_1|\,dx\right)^{1/2}}\\
&\le\frac{\|\tilde u_1\|_{L^\infty}}{1-2\int u_1|\tilde u_1-u_1|\,dx}\le\frac{\|\tilde u_1\|_{L^\infty}}{1-2\|u_1\|_{L^\infty}\|\tilde u_1-u_1\|_{L^1}}\le \frac{\delta}{1-2\delta\eps_1}\le  2\delta,
\end{split}
\end{equation}
for $\eps_1$ as above.
Suppose now that the claim holds for $1,\dots,k-1$. In order to prove the estimate for $v_k$ we first estimate the $L^1$ distance from $u_k$ to the orthogonalized function 
\begin{equation*}
w_k:=\begin{cases} \tilde u_1,\quad\text{if}\quad k=1,\\
\tilde u_k-\sum_{i=1}^{k-1}\Big(\int \tilde u_k v_i\,dx\Big)v_i,\quad\text{if}\quad k>1.\end{cases}
\end{equation*}

We first estimate $\|u_k-w_k\|_{L^1}$, that gives:

\begin{equation}\label{prfseconmainest08}
\begin{split}
\left\|u_k-w_k\right\|_{L^1}
&\le \left\|u_k-\tilde u_k\right\|_{L^1}+\sum_{i=1}^{k-1}\left|\int \tilde u_k v_i\,dx\right|\left(\|u_i\|_{L^1}+\|v_i-u_i\|_{L^1}\right) \\
&\le \eps_k+\sum_{i=1}^{k-1}\left|\int \tilde u_k v_i\,dx\right| \big(|\Omega|^{1/2}+\eps_{k-1}\big) \\
&\le\eps_k+\sum_{i=1}^{k-1}\left|\int (\tilde u_k-u_k) u_i+(v_i-u_i)u_k+(\tilde u_k-u_k)(v_i-u_i)\,dx\right| \big(|\Omega|^{1/2}+\eps_{k-1}\big) \\
&\le \eps_k+\sum_{i=1}^{k-1}\left(\|\tilde u_k-u_k\|_{L^1} \|u_i\|_{L^\infty}+\|v_i-u_i\|_{L^1}\|u_k\|_{L^\infty}+\|\tilde u_k-u_k\|_{L^1}\|v_i-u_i\|_{L^\infty}\right) \big(|\Omega|^{1/2}+\eps_{k-1}\big)\\
&\le \eps_k+\big( (k-1)\eps_k\delta+C_{k-1}\eps_{k-1}\delta+(k-1)\eps_k C_{k-1}\big) \big(|\Omega|^{1/2}+\eps_{k-1}\big)\\
&\le \Big[1+\big(|\Omega|^{1/2}+\overline\eps_{k-1}\big)\big((k-1)\delta+C_{k-1}\delta+(k-1)C_{k-1}\delta\big)\Big]\eps_k.
\end{split}
\end{equation}
Then we deal with $\|w_k\|_{L^{\infty}}$:
\begin{equation}\label{prfseconmainest0803}
\begin{split}
\left\|w_k\right\|_{L^\infty}&\le \left\|\tilde u_k\right\|_{L^\infty}+\sum_{i=1}^{k-1}\left|\int \tilde u_k v_i\,dx\right|\|v_i\|_{L^\infty} \\
&\le \delta+C_{k-1}\sum_{i=1}^{k-1}\left|\int (\tilde u_k-u_k) u_i+(v_i-u_i)u_k+(\tilde u_k-u_k)(v_i-u_i)\,dx\right| \\
&\le \delta +C_{k-1}\big((k-1)\delta+C_{k-1}\delta+(k-1)C_{k-1}\delta\big)\eps_k\\
&\le \delta \Big(1+C_{k-1}\big((k-1)+C_{k-1}+(k-1)C_{k-1}\big)\Big).
\end{split}
\end{equation}
We set for simplicity $\tilde C_k$ to be the largest of the constants appearing on the right hand side of \eqref{prfseconmainest08} and \eqref{prfseconmainest0803}. Thus we have 
$$\|u_k-w_k\|_{L^1}\le \tilde C_k\eps_k\qquad\text{and}\qquad \|w_k\|_{L^\infty}\le \tilde C_k.$$
Recalling that $v_k=\|w_k\|_{L^2}^{-1}w_k$ we have
\begin{equation}\label{prfseconmainest0804}
\begin{split}
\big|\left\|w_k\right\|_{L^2}-1\big|&\le \big|\left\|w_k\right\|_{L^2}^2-1\big|=\big|\left\|u_k+(w_k-u_k)\right\|_{L^2}^2-1\big|\\
&= \left|2\int_{\R^d}u_k(u_k-w_k)\,dx+\int_{\R^d}(u_k-w_k)^2\,dx\right|\\
&\le 2\|u_k\|_{L^\infty}\|u_k-w_k\|_{L^1}+\|u_k-w_k\|_{L^\infty}\|u_k-w_k\|_{L^1}\\
&\le 2\delta\tilde C_{k}\eps_k+(\delta+\tilde C_k)\tilde C_{k}\eps_k.
\end{split}
\end{equation}
We ask then that $\eps_k\le \overline\eps_k:=\frac12\Big(2\delta\tilde C_{k}+(\delta+\tilde C_k)\tilde C_{k}\Big)^{-1}$. Thus, $1/2\le \|w_k\|_{L^2}\le 3/2$ and we have the estimate
\begin{equation*}
\begin{split}
\|v_k\|_{L^\infty}=\|w_k\|_{L^2}^{-1}\left\|w_k\right\|_{L^\infty}&\le 2\tilde C_k.
\end{split}
\end{equation*}
On the other hand, repeating precisely the same procedure as in \eqref{firstindstepest08} we obtain
\begin{equation*}
\begin{split}
\|u_k-v_k\|_{L^1}&\le \|u_k-w_k\|_{L^1}+\|w_k-v_k\|_{L^1}\le \tilde C_k\eps_k+\frac{\big|\left\|w_k\right\|_{L^2}^2-1\big|}{\|w_k\|^2_{L^2}}\|w_k\|_{L^1}\\
&= \tilde C_k\eps_k+\frac{\big|\left\|u_k+(w_k-u_k)\right\|_{L^2}^2-1\big|}{\|u_k+(w_k-u_k)\|^2_{L^2}}\|u_k+(w_k-u_k)\|_{L^1}\\
&\le \big(1+12\tilde C_k|\Omega|^{1/2}\big)\tilde C_{k}\eps_k,
\end{split}
\end{equation*}
for $\eps_k\le \overline\eps_k$ , where $\overline\eps_k>0$ is small enough and depends on $\tilde C_k$, $\delta$ and $|\Omega|$. We conclude the recursive step and the proof of the claim by defining 
$$C_k:=2\big(1+12\tilde C_k|\Omega|^{1/2}\big)\tilde C_{k}.$$

We are now in position to prove \eqref{maingradest08} by induction. For $k=1$ we repeat the estimate from \eqref{firstindstepest09} and we get
\begin{equation*}
\begin{split}
\|\nabla v_1\|_{L^2}=\frac{\|\nabla \tilde u_1\|_{L^2}}{\|\tilde u_1\|_{L^2}}\le \frac{\|\nabla \tilde u_1\|_{L^2}}{1-2\|u_1\|_{L^\infty}\|u_1-\tilde u_1\|_{L^1}}\le(1+4\delta\eps_1)\|\nabla \tilde u_1\|_{L^2},
\end{split}
\end{equation*}
For $k>1$, by \eqref{prfseconmainest0804} we obtain
\begin{equation*}
\begin{split}
\left\|\nabla v_k\right\|_{L^2}&=\frac{\left\|\nabla w_k\right\|_{L^2}}{\|w_k\|_{L^2}}\le\frac1{1-\big|\| w_k\|_{L^2}-1\big|} \left\|\nabla \tilde u_k-\sum_{i=1}^{k-1}\Big(\int \tilde u_k v_i\,dx\Big)\nabla v_i\Big)\right\|_{L^2} \\
&=\Big(1+2\big(2\delta\tilde C_{k}+(\delta+\tilde C_k)\tilde C_{k}\big)\eps_k\Big)\Big(\left\|\nabla \tilde u_k\right\|_{L^2}+\sum_{i=1}^{k-1}\Big|\int \tilde u_k v_i\,dx\Big|\|\nabla v_i\|_{L^2}\Big),
\end{split}
\end{equation*}
Using one more time the estimate
$$\sum_{i=1}^{k-1}\Big|\int \tilde u_k v_i\,dx\Big|\le \big((k-1)\delta+C_{k-1}\delta+(k-1)C_{k-1}\delta\big)\eps_k,$$
from \eqref{prfseconmainest08}, and the inductive hypothesis we obtain the claim.  
\end{proof}

\begin{proof}[\bf Proof of Proposition \ref{quasimin_brno}] 
Let $\widetilde U\in H^1(\R^d;\R^k)$ be a vector-valued function satisfying the assumptions of Proposition \ref{quasimin_brno} and let $V=(v_1,\dots,v_k)\in H^1(\R^d;\R^k)$ be the function obtained through the orthonormalization procedure in Lemma \ref{lemmaorto} starting from $\widetilde U$. By Lemma \ref{laglem} we have that $U$ is a solution of \eqref{sommeV} and since we $\int v_iv_j\,dx=\delta_{ij}$ we can use $V$ as a test function in \eqref{sommeV} obtaining
\begin{align*}
\int_{\R^d}|\nabla U|^2\,dx+\Lambda|\{|U|>0\}|&\le \int_{\R^d}|\nabla V|^2\,dx+\Lambda|\{|V|>0\}|\\
&\le \left(1+\overline C_k\|U-\widetilde U\|_{L^1}\right)\int_{\R^d}|\nabla \widetilde U|^2\,dx+\Lambda|\{|\widetilde U|>0\}|,
\end{align*}
where the last inequality follows by Lemma \ref{lemmaorto} and the fact that by the construction of $V$ we have that $\{|V|>0\}\subset\{|\widetilde U|>0\}$. 
\end{proof}

\subsection{Non-degeneracy of the eigenfunctions}\label{sec:nondeg2}

The following Lemma will be applied to the case when $U$ is the vector of eigenfunctions on an optimal set, but it holds for functions   $U=(u_1,\dots,u_k)$ satisfying the quasi-optimality condition~\eqref{Uoptcond} or, more generally, to functions satisfying the following condition \eqref{intoptcondoss} which are roughly speaking subsolutions of \eqref{Uoptcond} since they are minimal only with respect to perturbations $\tilde U$ such that $|\tilde U|\le |U|$.  
\begin{equation}\label{intoptcondoss}
\begin{split}
&\text{There are constants $\qm>0$ and $\eps>0$ such that}\\ 
&\int_{\R^d}|\nabla U|^2\,dx+\Lambda\big|\{|U|>0\}\big|\le \Big(1+\qm\|U-\widetilde U\|_{L^1}\Big)\int_{\R^d}|\nabla \widetilde U|^2\,dx+\Lambda\big| \{|\widetilde U|>0\}\big|,\\
&\text{for every}\quad \widetilde U\in H^1(\R^d;\R^k)\quad\text{such that}\quad |\widetilde U|\le |U| \quad\text{and}\quad \|U-\widetilde U\|_{L^1}\le \eps. 
\end{split}
\end{equation}

\begin{lemma}[Non-degeneracy of $U$]\label{nondegcaflemma}
Let $U=(u_1,\dots,u_k)\in H^1(\R^d;\R^k)$ be a function satisfying the quasi-optimality condition \eqref{intoptcondoss}.
There are contants $c_0>0$ and $r_0>0$, depending on $d$, $K$, $\Lambda$, $\eps$ and $\|\nabla U\|_{L^2(\R^d;\R^k)}$, such that for every $x_0\in\R^d$ and $r\in(0,r_0]$ the following implication holds
\begin{equation*}
\Big(\|U\|_{L^\infty(B_{2r})}<c_0 r\Big)\Rightarrow \Big(U\equiv 0\ \ \text{in}\ \ B_r(x_0)\Big).
\end{equation*}
\end{lemma}
\begin{proof}
Suppose for simplicity $x_0=0$. Let $r>0$ be such that $r\le r_0$ and $\|U\|_{L^\infty(B_{2r})}\le c_0r$ with $c_0$ and $r_0$ that will be chosen later in \eqref{choice1} and \eqref{choice2}. 

Consider the radial functions 
$$\psi:B_{2}\setminus B_{1}\to\R\qquad\text{and}\qquad \phi: B_{2}\setminus B_{1}\to\R,$$
solutions of the PDEs
$$\ \ \Delta \psi=0\quad\text{in}\quad B_{2}\setminus B_{1},\qquad \psi=0\quad\text{on}\quad \partial B_{1},\qquad  \psi=1\quad\text{on}\quad \partial B_{2},$$
$$-\Delta \phi=1\quad\text{in}\quad B_{2}\setminus B_{1},\qquad \phi=0\quad\text{on}\quad \partial B_{1},\qquad  \phi=0\quad\text{on}\quad \partial B_{2}.$$
We set $\alpha=c_0 r>0$, while $\beta>0$ will be chosen in \eqref{choiceofbeta} and will also depend on $r>0$. 
We consider the function 
$$\eta(x)=\alpha\psi(x/r)+\beta r^2 \phi(x/r),$$
solution of the boundary value problem 
$$-\Delta \eta=\beta\quad\text{in}\quad B_{2r}\setminus B_{r},\qquad \eta=0\quad\text{on}\quad \partial B_{r},\qquad  \eta=\alpha\quad\text{on}\quad \partial B_{2r},$$
and we notice that we have the estimate
$$ |\nabla \eta|\le C_d\left(\beta r+\frac{\alpha}{r}\right)\le C_d\left(\beta r_0+c_0\right)\quad\text{on}\quad \partial B_r.$$
Consider the test function 
$$\tilde U=(\tilde u_1,\dots, \tilde u_k):\R^d\to\R^k,$$
defined by 
$$\tilde u_i=\begin{cases}u_i^+\wedge \eta-(u_i^-\wedge \eta)\quad\text{in}\quad B_{2r},\\
u_i\quad\text{in}\quad \R^d\setminus B_{2r}.
\end{cases}$$
We first choose $r_0$ and $c_0$ such that 
\begin{equation}\label{choice1}
\omega_d\, 2^d\,r_0^{d+1}\,c_0\le \eps,
\end{equation}
in such a way that $\|U\|_{L^1(B_{2r_0})}\le \eps$
and we define $\eps(2r)$ as
$$\eps(2r)=\sum_{i=1}^k\int_{\R^d}|u_i-\tilde u_i|\,dx=\sum_{i=1}^k\int_{B_{2r}}\big((u_i^+-\eta)_++(u_i^--\eta)_+\big)\,dx.$$
By \eqref{choice1} we have $\eps(2r)\le\eps(2r_0)\le \eps$ and so the optimality of $U$ gives 
\begin{align*}
\int_{\R^d}|\nabla U|^2\,dx+\Lambda|\{|U|>0\}|\le (1+\qm\eps(2r))\int_{\R^d}|\nabla \tilde U|^2\,dx+\Lambda|\{|\tilde U|> 0\}|.
\end{align*}
Since $U=\tilde U$ on $\R^d\setminus B_{2r}$ we have 
\begin{equation*}
\begin{split}
\int_{B_{r}}|\nabla U|^2\,dx+\Lambda|\{|U|>0\}\cap B_{r}|&\le \int_{B_{2r}\setminus B_{r}}\Big(|\nabla \tilde U|^2-|\nabla U|^2\Big)\,dx+\qm\eps(2r)\int_{\R^d}|\nabla \tilde U|^2\,dx\\
&= (1+\qm \eps(2r))\int_{B_{2r}\setminus B_{r}}\Big(|\nabla \tilde U|^2-|\nabla U|^2)\,dx+\qm \eps(2r)\int_{\R^d}|\nabla U|^2\,dx\\
&= (1+\qm\eps(2r))\int_{B_{2r}\setminus B_{r}}\Big(-|\nabla (\tilde U-U)|^2+2\nabla \tilde U\cdot\nabla (\tilde U-U)\Big)\,dx\\
&\qquad+\qm\eps(2r)\int_{\R^d}|\nabla U|^2\,dx\\
&\le 2(1+\qm\eps(2r))\int_{B_{2r}\setminus B_{r}}\nabla \tilde U\cdot\nabla (\tilde U-U)\,dx+\qm\eps(2r)\int_{\R^d}|\nabla U|^2\,dx.
\end{split}
\end{equation*}
We now estimate the first term in the right-hand side 
\begin{equation*}
\begin{split}
\int_{B_{2r}\setminus B_{r}}\nabla \tilde U\cdot\nabla (\tilde U-U)\,dx&=\sum_{i=1}^k\Big(\int_{B_{2r}\setminus B_{r}}\nabla \tilde u_i^+\cdot\nabla (\tilde u_i^+-u_i^+)\,dx+\int_{B_{2r}\setminus B_{r}}\nabla \tilde u_i^-\cdot\nabla (\tilde u_i^--u_i^-)\,dx\Big)\\
&= -\sum_{i=1}^k\Big(\int_{B_{2r}\setminus B_{r}}\nabla \eta\cdot\nabla (u_i^+-\eta)_+\,dx+\int_{B_{2r}\setminus B_{r}}\nabla \eta\cdot\nabla (u_i^--\eta)_+\,dx\Big)\\
&= -\sum_{i=1}^k\Big(\int_{B_{2r}\setminus B_{r}}\beta (u_i^+-\eta)_+\,dx+\int_{B_{2r}\setminus B_{r}}\beta (u_i^--\eta)_+\,dx\Big)\\
&\qquad+C_d\left(\beta r+\frac{\alpha}{r}\right)\sum_{i=1}^k\int_{\partial B_{r}}|u_i|\,d\HH^{d-1}.
\end{split}
\end{equation*}
{We now choose 
\begin{equation}\label{choiceofbeta}
\beta= \frac{\qm}{2(1+\qm\eps(2r))}\int_{\R^d}|\nabla U|^2\,dx,
\end{equation}
and we set 
\begin{equation*}
E(U,B_r)=\int_{B_{r}}|\nabla U|^2\,dx+\Lambda|\{U\neq0\}\cap B_{r}|.
\end{equation*}
Thus, we obtain the inequality 
\begin{equation}\label{mainenergy08}
\begin{split}
E(U,B_{r})&\le -2(1+\qm\eps(2r))\sum_{i=1}^k\Big(\int_{B_{2r}\setminus B_{r}}\beta (u_i^+-\eta)_+\,dx+\int_{B_{2r}\setminus B_{r}}\beta (u_i^--\eta)_+\,dx\Big)\\
&\qquad +C_d\left(\beta r+\frac{\alpha}{r}\right)\sum_{i=1}^k\int_{\partial B_{r}}|u_i|\,d\HH^{d-1}
 +\qm\eps(2r)\int_{\R^d}|\nabla U|^2\,dx\\
 &\leq C_d \left(\beta r+\frac{\alpha}{r}\right)\sum_{i=1}^k\int_{\partial B_{r}}|u_i|\,d\HH^{d-1}+K\|\nabla U\|_{L^2(\R^d;\R^k)}^2\sum_{i=1}^k\int_{ B_{r}}|u_i|\,dx,\\
\end{split}
\end{equation}
since, thanks to the choice of $\beta>0$ and the fact that $\eta=0$ in $B_r$, we have 
\[
\eps(2r)-\sum_{i}\Big(\int_{B_{2r}\setminus B_{r}} (u_i^+-\eta)_+\,dx+\int_{B_{2r}\setminus B_{r}} (u_i^--\eta)_+\,dx\Big)=\sum_{i=1}^k\int_{B_r}|u_i|\,dx.
\]
We now aim to estimate the term in the right hand side of~\eqref{mainenergy08} by $E(U,B_{r})$. 
By the $W^{1,1}$ trace inequality in $B_{r}$ we have 
\begin{equation*}
\begin{split}
\int_{\partial B_{r}} |u_i|\,d\HH^{d-1}&\le C_d\Big(\int_{B_{r}}|\nabla u_i|\,dx+\frac1r\int_{B_{r}}|u_i|\,dx\Big)\\
&\le C_d\Big(\frac12\int_{B_{r}}|\nabla u_i|^2\,dx+\frac12|\{|u_i|>0\}\cap B_{r}|\Big)+\frac{C_d}r c_0r|\{|u_i|>0\}\cap B_{r}|\\
&\le  C_d \,(1+c_0)\max\left\{1,{1/\Lambda}\right\}E(U,B_{r}).
\end{split}
\end{equation*}
Summing the above inequality for $i=1,\dots,k$ we get 
$$\sum_{i=1}^k\int_{\partial B_{r}} |u_i|\,d\HH^{d-1}\le k\, C_d \,(1+c_0)\max\left\{1,{1/\Lambda}\right\}E(U,B_{r})=:C_{k,d,\Lambda,c_0} \,E(U,B_{r}).$$
Since the above inequality holds also for every $s\in (0,r]$ we get
$$\sum_{i=1}^k\int_{B_{r}} |u_i|\,dx=\sum_{i=1}^k\int_0^r\,ds\int_{\partial B_{s}} |u_i|\,d\HH^{d-1}\le  C_{k,d,\Lambda,c_0} \int_0^r\,ds\, E(U,B_{s})\le r\,C_{k,d,\Lambda,c_0} \, E(U,B_{r}) .$$
We can finally estimate the right hand side of \eqref{mainenergy08} obtaining 
\begin{equation*}
\begin{split}
E(U,B_r)&\le C_{k,d,\Lambda,c_0}\left(\beta r+\frac\alpha{r}+rK\|\nabla U\|_{L^2(\R^d;\R^k)}^2\right) E(U,B_r)\\
&\le C_{k,d,\Lambda,c_0}\left(2K\|\nabla U\|_{L^2(\R^d;\R^k)}^2 r_0+c_0\right) E(U,B_r).
\end{split}
\end{equation*}
}
Choosing $r_0$ and $c_0$ such that 
\begin{equation}\label{choice2}
k\, C_d \,(1+c_0)\max\left\{1,{1/\Lambda}\right\}\left(2K\|\nabla U\|_{L^2(\R^d;\R^k)}^2 r_0+c_0\right)<1,
\end{equation}
for a dimensional constant $C_d>0$, we get that $E(U,B_r)=0$ and so we obtain the claim.
\end{proof}

\begin{oss}[Subharmonicity of $|U|$]\label{Usub}
Let $\Omega\subset\R^d$ be an open set of finite measure and $u_1,\dots,u_k$ be the first $k$ normalized eigenfunctions on $\Omega$. Then the function $|U|=|(u_1,\dots,u_k)|$ satisfies, weakly in $H^1(\R^d)$, the inequality 
$$\Delta |U|+\lambda_k(\Omega)|U|\ge 0\quad\text{in}\quad\R^d,\qquad |U|\in H^1_0(\Omega).$$
In fact, on the set $\omega:=\{|U|>0\}$, $|U|$ satisfies the inequality
\begin{align*}
\Delta |U|&=\sum_j \left[\frac{u_j\Delta u_j}{|U|}+\frac{|\nabla u_j|^2}{|U|}-\frac{u_j\nabla u_j\cdot\nabla|U|}{|U|^2}\right]\\
& =-\frac{1}{|U|}\sum_j\lambda_j(\Omega) u_j^2+\frac1{|U|^3}\sum_{i,j} \Big(u_i^2 |\nabla u_j|^2-u_iu_j\nabla u_i\cdot\nabla u_j\Big)\ge -\lambda_k(\Omega)|U|,
\end{align*}
while the result on the entire space follows from the fact that $|U|$ is positive.
\end{oss}

\begin{oss}[Equivalent definitions of non-degeneracy]
Suppose that $u\in H^1(B_R)$ is such that:
\begin{enumerate}
\item $u\ge 0$ and $\Delta u+1\ge 0$ weakly in $H^1_0(B_R)$.
\item There are constants $c_0$ and $r_0$ such that for all $r\leq r_0$,
$$\Big(\|u\|_{L^\infty(B_{2r})}\le c_0 r\Big)\Rightarrow\Big(u\equiv 0\quad\text{in}\quad B_r\Big).$$ 
\end{enumerate}
Then there are constants $r_1$ and $c_1$, depending only on the dimension $d$ and the constants $c_0$ and $r_0$, such that the following implication hold for every $r\le r_1$:
$$\Big(\mean{B_r}{u\,dx}\le c_1 r\Big)\Rightarrow\Big(u\equiv 0\quad\text{in}\quad B_{r/4}\Big),$$ 
$$\Big(\mean{\partial B_r}{u\,d\HH^{d-1}}\le c_1 r\Big)\Rightarrow\Big(u\equiv 0\quad\text{in}\quad B_{r/4}\Big).$$ \end{oss}

\subsection{Density estimate and non-degeneracy of the first eigenfunction}
First of all we prove a non-degeneracy result for the gradient, which will lead to a non-degeneracy for $u_1$.
\begin{lemma}[Non-degeneracy of $|\nabla U|$]\label{nondegdu}
Let $\Omega$ be an optimal set for problem~\eqref{introP} and let $U=(u_1,\dots,u_k)$ be the vector of the first $k$ normalized eigenfunctions. Then there are constants $c>0$ and $r>0$ such that
\begin{equation}\label{nondeggradU}
\big|\nabla U\big|^2=\sum_{j=1}^k|\nabla u_j|^2\ge c\quad\text{on the set}\quad S_r:=\big\{x\in\Omega\ :\ \text{dist}(x,\partial\Omega)\le r\big\}.
\end{equation} 
\end{lemma}
\begin{proof}
The key point of our proof is that there are constants $\overline c>0$ and $\overline r>0$ such that 
\begin{equation}\label{claimnondeg}
\overline c\le \mean{B_\rho(x_0)}{\big|\nabla U\big|^2\,dx},\quad\text{where}\quad \rho=\text{dist}(x_0,\partial\Omega)\le \overline r.
\end{equation}
We prove this starting from the non-degeneracy of $U$, which implies (applying an H\"older inequality) that for all $r\le r_0$ and for some constant $c$, \[
\int_{B_r\cap \Om}{|U|^2}\geq c r^{d+2}.
\]
For all $j=1,\dots, k$ we consider $u^{\pm}_j$ and we call $h_j^\pm$ their harmonic extension of $u_j^\pm$ in $B_r$.
For all $j=1,\dots, k$, we can deduce, using also the Poincar\'e inequality, \[
\begin{split}
  &\frac{c}{r^2}\int_{B_r\cap \Om}{(u_j^\pm)^2\,dx}\leq \frac{c}{r^2}\int_{B_r}{(u_j^\pm-h_j^\pm)^2\,dx}\leq \int_{B_r}{|\nabla (u_j^\pm-h_j^\pm)|^2\,dx}\\
  &=\int_{B_r}{|\nabla u_j^\pm|^2-|\nabla h_j^\pm|^2\,dx}\leq \int_{B_r}{|\nabla u_j^\pm|^2\,dx},
\end{split}
\]
for some constant $c$.
Then summing up over $j$ and using the non-degeneracy of $U$, we obtain \[
\int_{B_r}{|\nabla U|^2}=\sum_{j=1}^k{\int_{B_r}{|\nabla u^+_j|^2+|\nabla u^-_j|^2}}\geq \frac{c_1}{r^2}\int_{B_r\cap \Om}{|U|^2}\geq c_2 r^d,
\] 
for some constants $c_1, c_2$. This easily implies the claim~\eqref{claimnondeg}.

Then, for every $x_0\in S_{\overline r}$ there is $j\in\{1,\dots,k\}$ and $e\in\{e_1,\dots,e_d\}$ such that $\ds c_0\le \mean{B_\rho(x_0)}{\nabla_e u_j\,dx}$. On the other hand, on the ball $B_\rho(x_0)\subset \Omega$, the function $v=\nabla_e u_j$ satisfies the equation $-\Delta v=\lambda_j(\Omega)v$ and so we have 
$$|\Delta v|\le \lambda_k(\Omega)L,$$
where $L$ denotes the Lipschitz constant of $U$.
Thus, by the subharmonicity of $v(x)+|x-x_0|^2\lambda_k(\Omega)L$ we have 
$$|\nabla u_j|\ge \nabla_e u_j\ge \mean{B_\rho(x_0)}{\nabla_e u_j\,dx}-\rho^2\lambda_k(\Omega)L\ge c_0-r_0^2\lambda_k(\Omega)L,$$
which concludes the proof.
\end{proof}

It is important to highlight that, until now, we needed as hypothesis on $U$ only a quasi-minimality condition~\eqref{Uoptcond} and no sign assumption on the $u_i$ was involved. On the other hand, in the next lemmas, it will become essential that the first component $u_1$ of the vector $U$ is positive.
\begin{lemma}[Non-degeneracy of $u_1$]\label{nondegu1}
Suppose that $\Omega$ is a connected optimal set for problem~\eqref{introP}. Then there is a constant $C>0$ such that $C u_1\ge |U|$ on $\Omega$.
\end{lemma}
\begin{proof}
Let $r$ and $c$ be as in \eqref{nondeggradU}. 
Consider the function $v=|U|+|U|^2/2$. On the strip $S_r$ we have 
$$\Delta v=\Delta |U|+\sum_{j=1}^k (|\nabla u_j|^2+u_j\Delta u_j)\ge c_0-\lambda_k(\Omega)(|U|+|U|^2).$$
Since $|U|$ is continuous and $0$ on $\partial\Omega$ we have that there is $r>0$ such that $v$ is subharmonic on the strip $S_{r}$.

Let $\Omega_r=\{x\in\Omega\ :\ \text{dist}(x,\partial\Omega)\ge r\}$. Since $\Omega$ is connected we have that $\inf_{x\in\Omega_r}u_1>0$ and so there is $M>0$ such that $Mu_1\ge v$ on $\Omega_r$. On the other hand $u_1$ is superharmonic on $S_r$ which gives that $Mu_1\ge v\ge |U|$ on the entire domain $\Omega$. 
\end{proof}

The last lemma of this Section provides a density estimate for the optimal set $\Om_U$. We remark that, in order to obtain the upper bound on the density, it is fundamental to know that $u_1$ is non-negative and non-degenerate: without this assumption we are not able to prove such a claim. 
Here is the main difficulty if one wants to prove an extension of the Alt-Caffarelli result to the vectorial case in the general setting.
\begin{lemma}[Density estimate for $\Om_U$]\label{densestlemma}
	Suppose that $U\in C(B_{R};\R^k)$ is a Lipschitz continuous function satisfying the quasi-minimality condition \eqref{Uoptcond}. Then there are constant $r_0>0$ and $\eps_0>0$ such that 
	\begin{equation}\label{densest}
	\eps_0|B_r|\le \big|\{|U|>0\}\cap B_r(x_0)\big|\le (1-\eps_0)|B_r|,\quad\text{for every}\quad x_0\in \partial\{|U|>0\}\quad\text{and}\quad r\le r_0.   
	\end{equation}
\end{lemma}
\begin{proof}
	The proof follows by the same argument as in \cite{altcaf}. We assume that $x_0=0$. By Lemma \ref{nondegcaflemma} we have that for $r$ small enough $\ds\|U\|_{L^\infty(B_{r/2})}\ge \frac{c_0 r}4$. Thus there is some $x_r\in B_{r/2}$ such that $\ds|U|(x_r)\ge \frac{c_0 r}4$. On the other hand $|U|$ is Lipschitz continuous, and so, setting $\ds\theta=\inf\Big\{\frac12,\frac{c_0}{4\|\nabla |U|\|_{L^\infty}}\Big\}$ we have that $|U|>0$ on $B_{\theta r}(x_r)$ and this proves the lower bound in \eqref{densest}.\\
	
	For the upper bound, we notice that since $0\in\partial\{|U|>0\}$ we can apply Lemma~\ref{nondegu1} obtaining that there are constants $c_1$ and $r_0$ such that 
	$$\ds \mean{\partial B_r}{u_1\,d\HH^{d-1}}\ge c_1 r\quad \text{for every}\quad r\le r_0.$$  
Let $\widetilde U=(\tilde u_1,\dots,u_k)$, where $\tilde u_1$ is the harmonic extension of $u_1$ on the ball $B_r$. By the quasi-optimality of $U$ we have 
	\begin{equation}\label{qopteqdenslem}
	\begin{split}
	\Lambda\big| \{|U|=0\}\cap B_r\big|
	&\ge \int_{B_R}|\nabla U|^2\,dx-\Big(1+\qm\|U-\widetilde U\|_{L^1}\Big)\int_{B_R}|\nabla \widetilde U|^2\,dx\\
	&\ge \int_{B_r}|\nabla (U-\widetilde U)|^2\,dx-\qm\|U-\widetilde U\|_{L^1}\int_{B_R}|\nabla U|^2\,dx\\
	&=\int_{B_r}|\nabla (u_1-\widetilde u_1)|^2\,dx-\qm\|u_1-\widetilde u_1\|_{L^1}\int_{B_R}|\nabla U|^2\,dx.
	\end{split}
	\end{equation}
Let $L=\|\nabla u_1\|_{L^\infty}$. Then $\|u_1\|_{L^\infty(B_r)}\le Lr$ and by the maximum principle $\|\widetilde u_1\|_{L^\infty(B_r)}\le Lr$. Thus we have the estimate 
\begin{equation}\label{ungranbelfilm1}
\qm\|u_1-\widetilde u_1\|_{L^1}\int_{B_R}|\nabla U|^2\,dx\le \omega_d\qm L\|\nabla U\|_{L^2(B_R)}r^{d+1}=:Cr^{d+1}.
\end{equation}
In order to estimate $\ds\int_{B_r}|\nabla (u_1-\widetilde u_1)|^2\,dx$ we first notice that by the Poincar\'e inequality in $B_r$ we have 
\begin{equation}\label{granbelfilm2}
\int_{B_r}|\nabla (u_1-\widetilde u_1)|^2\,dx\ge \frac{\lambda_1(B_1)}{r^2}\int_{B_r}| u_1-\widetilde u_1|^2\,dx\ge \frac{C_d}{|B_r|}\Big(\frac1r\int_{B_r}|u_1-\widetilde u_1|\,dx\Big)^2.
\end{equation}
Let $\kappa\in(0,1/3)$. Since $\widetilde u_1$ is non-negative and harmonic in $B_r$ the Harnack inequality for $\widetilde u_1$ together with the non-degeneracy of $u_1$ gives that
	$$c_1r\le \mean{\partial B_r}{ u_1\,d\HH^{d-1}} =\mean{\partial B_r}{\widetilde u_1\,d\HH^{d-1}}=\widetilde u_1(0)\le\max_{B_{\kappa r}}\widetilde u_1\le \left(\frac{1-\kappa}{1-3\kappa}\right)^d \min_{B_{\kappa r}}\widetilde u_1,$$
while the Lipschitz continuity of $u_1$ gives that  $\ds \max_{B_{\kappa r}} u_1\le L\kappa r$.
Thus for $\kappa$ small enough (depending on $d$, $c_1$ and $L$) we have 
$$\widetilde u_1\ge |u_1|+\frac{c_1}{3}r\quad\text{in}\quad B_{\kappa r}.$$
	Together with \eqref{qopteqdenslem}, \eqref{ungranbelfilm1} and \eqref{granbelfilm2} this gives
	$$\Lambda\big| \{|U|=0\}\cap B_r\big|
	\ge \frac{C_d}{|B_r|}\Big(\frac1r\int_{B_{\kappa r}}|u_1-\widetilde u_1|\,dx\Big)^2 - C r^{d+1}\ge C_d c_1^2\kappa^{2d} r^d-Cr^{d+1}\ge \frac{C_d c_1^2\kappa^{2d}}2 r^d, $$
for $r$ small enough.	
\end{proof}

\section{Weiss monotonicity formula}\label{ssmono}

In this section we establish a monotonicity formula in the spirit of \cite{weiss99}. 
Following the original notation from \cite{weiss99}, for a function $U\in H^1(\R^d;\R^k)$ we define 
\begin{equation}\label{weiss_fun}
\phi(U,x_0,r):=\frac{1}{r^{d}}\left(\int_{B_r(x_0)}|\nabla U|^2\,dx+\Lambda\big|\{|U|>0\}\cap B_r(x_0)\big|\right)-\frac{1}{r^{d+1}}\int_{\partial B_r(x_0)}|U|^2\,d\HH^{d-1}.
\end{equation}
The monotonicity of $\phi(U,x_0,\cdot)$ is related to the classification of the blow-up limits and is an essential tool for proving the regularity of the free boundary. The following proposition concerns the case when $U$ is the vector of the first $k$ eigenfunctions on an optimal set.
\begin{prop}[Monotonicity formula for the optimal eigenfunctions]\label{mono_weiss} Let $\Omega$ be optimal for \eqref{introP} and let $U=(u_1,\dots,u_k)\in H^1_0(\Omega;\R^k)$ be the vector of the first $k$ normalized eigenfunctions on $\Omega$. Suppose that $x_0\in \partial\Omega$. Then there are constants $r_0$ and $C_1$ such that the function $r\mapsto\phi(U,x_0,r)$ satisfies the following inequality for every $r\le r_0$ :
$$\frac{d}{dr}\phi(U,x_0,r)\ge \frac{1}{r^{d+2}}\sum_{i=1}^k\int_{\partial B_r(x_0)} |(x-x_0)\cdot\nabla u_i-u_i|^2\,dx -C_1.$$
Moreover, the limit $\ds \lim_{r\to 0^+}\phi(U,x_0,r)$ exists and is a real number. 
\end{prop}
The last result of the section concerns the vector-valued functions $U=(u_1,\dots,u_k)\in H^1_{loc}(\R^d;\R^k)$ that are local minimizers of the functional $\mathcal F_0(U)=\int |\nabla U|^2\,dx+|\{|U|>0\}|$ in the sense of the following definition.

\begin{deff}\label{locmin}
We say that a function $U\in H^1_{loc}(\R^d;\R^k)$ is a local minimizer (we note that this is sometimes called absolute or global minimizer) of the functional 
$$\mathcal F_0(U)=\int |\nabla U|^2\,dx+\Lambda|\{|U|>0\}|,$$
if for every $B_R\subset\R^d$ and for every function $\tilde U\in H^1_{loc}(\R^d;\R^k)$ such that $\tilde U-U\in H^1_0(B_R;\R^k)$ we have 
$$\int_{B_R}|\nabla U|^2\,dx+\Lambda\big|\{|U|>0\}\cap B_R\big|\le \int_{B_R}|\nabla \tilde U|^2\,dx+\Lambda\big|\{|\tilde U|>0\}\cap B_R\big|.$$
\end{deff}

\begin{prop}[Monotonicity formula for local minimizers of $\mathcal F_0$]\label{mono_weiss_global}  Suppose that $U=(u_1,\dots,u_k)\in H^1_{loc}(\R^d;\R^k)$ is a local minimizer of the functional $\mathcal F_0$ in sense of Definition \ref{locmin}. Then the function $\phi(r):=\phi(U,0,r)$ from \eqref{weiss_fun} satisfies the inequality
$$\phi'(r)\ge \frac{1}{r^{d+2}}\sum_{i=1}^k\int_{\partial B_r} |x\cdot\nabla u_i-u_i|^2\,dx.$$
If moreover, $\phi$ is constant in $(0,+\infty)$, then the function $U$ is one-homogeneous. \\
\end{prop} 

For the sake of simplicity in the rest of the section we will fix $x_0=0$ and $\phi(r):=\phi(U,0,r)$.

\noindent In order to prove Proposition \ref{mono_weiss} and Proposition \ref{mono_weiss_global} we need the following lemma, in which, following the ideas from \cite[Theorem 1.2]{weiss99}, we compare the function $U$ with its one-homogeneous extension in the ball $B_r$.
\begin{lemma}\label{lemmadorin}
Let $U\in H^1(\R^d;\R^k)\cap W^{1,\infty}(\R^d;\R^k)$ be a Lipschitz continuous function such that $U(0)=0$.
Suppose that $U$ is a quasi-minimizer of $\F_{\qm}$ in sense of \eqref{Uoptcond}. Then, there are constants $r_0>0$ and $C_0>0$ such that, for every $r\in(0,r_0)$, we have the estimate
\begin{equation}\label{dorinesttorin}
	\begin{split}
		\int_{B_r}|\nabla U|^2\,dx+\Lambda\big|B_r\cap \{|U|>0\}\big|
		&\le \frac{r}{d}\int_{\partial B_r}\left(|\nabla_\tau U|^2+\frac{|U|^2}{r^2}\right)\,d\HH^{d-1}\\
		&\qquad\qquad+\Lambda\frac{r}{d}\HH^{d-1}\Big(\partial B_r\cap \{|U|>0\}\Big)+C_0 r^{d+1}. 
	\end{split}
\end{equation}
\end{lemma}
\begin{proof}
Let $U=(u_1,\dots,u_k)$ be a quasi-minimizer in the sense of~\eqref{Uoptcond} with constants $K,\eps$ and we can clearly assume that $\|U\|_{L^\infty}\leq \eps^{-1}$. We consider the one homogeneous function $\widetilde U=(\widetilde u_1,\dots,\widetilde u_k):B_r\to\R^k$ defined by  $\ds \widetilde U(x):=\frac{|x|}{r}U\left(x\frac{r}{|x|}\right)$. For its components $\widetilde u_i$ we have $\ds \widetilde u_i(x):=\frac{|x|}{r}u_i\left(x\frac{r}{|x|}\right)$ and 
\begin{align*}
|\nabla \widetilde u_i|^2(x)=|\nabla_\tau u_i|^2\left(x\frac{r}{|x|}\right)+r^{-2}u_i^2\left(x\frac{r}{|x|}\right).
\end{align*}
Integrating over $B_r$ and summing for $i=1,\dots,k$ we obtain 
\begin{align*}
\int_{B_r}|\nabla \widetilde U|^2\,dx
=\sum_{i=1}^k\frac{r}{d}\int_{\partial B_r}\left(|\nabla_\tau u_i|^2+\frac{u_i^2}{r^2}\right)d\HH^{d-1}=\frac{r}{d}\int_{\partial B_r}\left(|\nabla_\tau U|^2+\frac{|U|^2}{r^2}\right)d\HH^{d-1},\end{align*}
while for the measure term we have that
\begin{align*}
\big|B_r\cap \{|\widetilde U|>0\}\big|=\frac{r}{d}\HH^{d-1}\Big(\partial B_r\cap \{|U|>0\}\Big).
\end{align*}	
Since $U\equiv \widetilde U$ on $\partial B_r$, the minimality of $U$ in $B_r$ gives 
\begin{equation*}
\begin{split}
\int_{B_r}|\nabla U|^2\,dx+\Lambda\big|B_r\cap \{|U|>0\}\big|&\le \int_{B_r}|\nabla \widetilde U|^2\,dx+\Lambda\big|B_r\cap \{|\widetilde U|>0\}\big|+\qm\|U-\widetilde U\|_{L^1}\int_{\R^d}{|\nabla \widetilde U|^2\,dx}\\
&\le \frac{r}{d}\int_{\partial B_r}\Big(|\nabla_\tau U|^2+\frac{|U|^2}{r^2}\Big)\,d\HH^{d-1}+\Lambda\frac{r}{d}\HH^{d-1}\Big(\partial B_r\cap \{|U|>0\}\Big)\\
&\qquad\qquad+\qm 2r|B_r|\|\nabla U\|_{L^\infty}\left(\int_{\R^d}{|\nabla U|^2\,dx}+2|B_r|\|\nabla U\|_{L^\infty}^2\right).
\end{split}	
\end{equation*}
It is now sufficient to choose $C_0$ and $r_0$ such that
$$2r_0^{d+1}|B_1|\|\nabla U\|_{L^\infty}\le \eps\qquad\text{and}\qquad C_0\ge 2\qm|B_1|\|\nabla U\|_{L^\infty}\left(\int_{\R^d}{|\nabla U|^2\,dx}+2|B_{r_0}|\|\nabla U\|_{L^\infty}^2\right),$$
where $\qm$ and $\eps$ are the constants from \eqref{Uoptcond}.
\end{proof}

We are now in position to prove the desired monotonicity formula for the function $\phi$.
\begin{proof}[\bf Proof of Proposition \ref{mono_weiss}] Let $r_0$ and $C_0$ be the constants from Lemma \ref{lemmadorin} and let $C_1=dC_0$. Calculating the derivative $\phi'(r)$ and using~\eqref{dorinesttorin} from Lemma~\ref{lemmadorin}, we have 
\begin{align*} 
\phi'(r)&=\frac1{r^{d}}\left(\int_{\partial B_r}|\nabla U|^2\,d\HH^{d-1}+\Lambda\HH^{d-1}(\{|U|>0\}\cap \partial B_r)\right)-\frac{d}{r^{d+1}}\left(\int_{B_r}|\nabla U|^2\,dx+\Lambda|\{|U|>0\}\cap B_r|\right)\\
&\quad +\frac{2}{r^{d+2}}\int_{\partial B_r}|U|^2\,d\HH^{d-1}-\frac{1}{r^{d+1}}\sum_{i=1}^k\int_{\partial B_r}2u_i\frac{\partial u_i}{\partial \nu}\,d\HH^{d-1}\\
&\ge \frac1{r^{d}}\left(\int_{\partial B_r}|\nabla U|^2\,d\HH^{d-1}+\Lambda\HH^{d-1}(\{|U|>0\}\cap \partial B_r)\right)\\
&\quad -\frac{d}{r^{d+1}}\left(\frac{r}{d}\int_{\partial B_r}\left(|\nabla_\tau U|^2+\frac{|U|^2}{r^2}\right)d\HH^{d-1}+\frac{r}{d}\Lambda\HH^{d-1}(\{|U|>0\}\cap \partial B_r)+C_0 r^{d+1}\right)\\
&\quad+\frac{2}{r^{d+2}}\int_{\partial B_r}|U|^2\,d\HH^{d-1}-\frac{1}{r^{d+1}}\sum_{i=1}^k\int_{\partial B_r}2u_i\frac{\partial u_i}{\partial \nu}\,d\HH^{d-1}\\
&=\frac{1}{r^d}\sum_{i=1}^k\int_{\partial B_r}\left|\frac{\partial u_i}{\partial \nu}\right|^2\,d\HH^{d-1}+\frac{1}{r^{d+2}}\int_{\partial B_r}|U|^2\,d\HH^{d-1}-\frac{1}{r^{d+1}}\sum_{i=1}^k\int_{\partial B_r}2u_i\frac{\partial u_i}{\partial \nu}\,d\HH^{d-1}-C_1\\
&=\frac{1}{r^{d+2}}\sum_{i=1}^k\int_{\partial B_r} \left(r^2\left|\frac{\partial u_i}{\partial \nu}\right|^2+u_i^2-2ru_i\frac{\partial u_i}{\partial \nu}\right)d\HH^{d-1}-C_1= \frac{1}{r^{d+2}}\int_{\partial B_r}{|x\cdot \nabla U-U|^2d\HH^{d-1}}-C_1,
\end{align*}
which concludes the proof of the first part of Proposition \ref{mono_weiss}. In particular, we obtain that the function $r\mapsto \phi(r)+C_1r$ is non-decreasing. Thus the limit $\ds\lim_{r\to0+}(\phi(r)+C_1r)=\lim_{r\to0+}\phi(r)$ does exist and is necessarily a real number or $-\infty$. In order to exclude this last possibility, we notice that  due to the Lipschitz continuity of $U$ and the fact that $U(0)=0$, we have that
$$\phi(r)\ge -\frac1{r^{d+1}}\int_{\partial B_r}|U|^2\,d\HH^{d-1}\ge -d\omega_d\|\nabla U\|_{L^\infty}^2,\quad\text{for every}\quad r>0,$$
which finally proves that $\ds\lim_{r\to0+}\phi(r)$ is finite.
\end{proof}

\begin{proof}[\bf Proof of Proposition \ref{mono_weiss_global}] We notice that if $U$ is a local minimizer of the functional $\mathcal F_0$, then both the constants $C_0$ and $C_1$ defined above can be taken equal to zero. The last claim of the proposition follows by the fact that if $\phi'\equiv 0$, then $x\cdot U\equiv U$ in $\R^d$, which proves that $U$ is $1$-homogeneous.
\end{proof}

\section{Blow-up sequences and blow-up limits}\label{sec:blowup}

Let $U:\R^d\to\R^k$ be a given Lipschitz function. For $r>0$ and $x\in\R^d$ such that $U(x)=0$, we define 
$$\ds U_{r,x}(y):=\frac1r U(x+ry).$$
When $x=0$ we will use the notation $U_r:=U_{r,0}$.

Suppose now that $(r_n)_{n\ge0}\subset\R^+$ and $(x_n)_{n\ge0}\subset\R^d$ are two sequences such that
\begin{equation}\label{rnxn}
\lim_{n\to\infty}r_n=0,\qquad \lim_{n\to\infty}x_n=x_0,\qquad x_n\in\partial\{|U|>0\}\quad\text{for every}\quad n\ge0.
\end{equation}
Then the sequence $\{U_{r_n,x_n}\}_{n\in\N}$ is uniformly Lipschitz and locally uniformly bounded in $\R^d$. Thus, up to a subsequence, $U_{r_n,x_n}$ converges locally uniformly in $\R^d$ as $n\to\infty$. 

\begin{deff}
Let $U:\R^n\to\R^k$ be a Lipschitz function, $r_n$ and $x_n$ be two sequences satisfying \eqref{rnxn}. 
\begin{itemize}
\item We say that the sequence $U_{r_n,x_n}$ is a blow-up sequence with variable center (this is sometimes called pseudo-blow-up). 
\item If the sequence $x_n$ is constant, i.e. $x_n=x_0$ for every $n\ge 0$, we say that the sequence $U_{r_n,x_0}$ is a blow-up sequence with fixed center. 
\item We denote by $\mathcal{BU}_U(x_0)$ the space of all the limits of blow-up sequences with fixed center $x_0$.
\end{itemize}
\end{deff}

The main result of this section is the following :
\begin{prop}[Structure of the blow-up limits]\label{blowup}
Let $\Omega$ be optimal for \eqref{introP} and let $U=(u_1,\dots,u_k)$ be the vector of the first $k$ normalized eigenfunctions on $\Omega$. For every $x_0\in \partial\Omega$ and $U_0\in\mathcal{BU}_U(x_0)$ there is a unit vector $\xi\in\partial B_1\subset\R^k$ such that $U_0=\xi|U_0|$. Moreover the (real-valued) function $|U_0|$ is not identically zero and satisfies the following properties:
\begin{enumerate}
\item $|U_0|$ is $1$-homogeneous ;
\item $|U_0|$ is a local minimizer (in the sense of Definition \ref{locmin}) of the Alt-Caffarelli functional 
\begin{equation*}
H^1_{loc}(\R^d;\R)\ni u\mapsto\int|\nabla u|^2\,dx+\Lambda|\{u>0\}|.
\end{equation*}
\end{enumerate}
\end{prop} 

The rest of the section is dedicated to the proof of Proposition \ref{blowup}. In Proposition \ref{blowupexistence} we prove that the blow-up sequences (of fixed or variable center) converge strongly in $H^1_{loc}$ and the corresponding free boundaries converge in the Hausdorff distance. In Lemma \ref{Uinftymin} we prove that the vector-valued function $U_0$ is a local minimizer (in the sense of Definition \ref{locmin}) of the functional 
\begin{equation*}
H^1_{loc}(\R^d;\R^k)\ni U\mapsto\int|\nabla U|^2\,dx+\Lambda|\{|U|>0\}|.
\end{equation*}
We apply then the Weiss monotonicity formula (Proposition \ref{mono_weiss} and Proposition \ref{mono_weiss_global}) to obtain the $1$-homogeneity of $U_0$, that we use to prove the existence of the vector $\xi$ in Lemma \ref{Uinftyhom2}. This result together with the optimality of $U_0$ gives the optimality of $|U_0|$, which is finally proved in Lemma \ref{Uinftyopt2}. \\

As a consequence of Proposition \ref{blowup} we get the following result. 

\begin{cor}\label{mainconnected}
Every optimal set for \eqref{introP} is connected.
\end{cor}
\begin{proof}
Let $\Omega$ be an optimal set for the problem~\eqref{introP}. Suppose that $\Om$ is a union of two disjoint open sets $\Om_1$ and $\Om_2$. Then the spectrum of $\Omega$ is given by $\sigma(\Omega)=\sigma(\Omega_1)\cup\sigma(\Omega_2)$ and in particular there is some $l\in{1,\dots,k-1}$ such that 
$$\{\lambda_1(\Omega),\dots,\lambda_k(\Omega)\}=\{\lambda_1(\Omega_1),\dots,\lambda_l(\Omega_1)\}\cup\{\lambda_1(\Omega_2),\dots,\lambda_{k-l}(\Omega_2)\}.$$
Now since $\Omega$ is optimal for the sum $\lambda_1+\dots+\lambda_k$, we have that $\Omega_1$ has to be optimal for $\lambda_1+\dots+\lambda_l$ and $\Omega_2$ for $\lambda_1+\dots+\lambda_{k-l}$. Let $\widetilde\Omega_1$ and $\widetilde\Omega_2$ be translations of $\Omega_1$ and $\Omega_2$ such that $\widetilde \Omega_1$ and $\widetilde\Omega_2$ are disjoint and tangent in $0\in\partial\widetilde\Omega_1\cap\partial\widetilde\Omega_2$. Setting $\widetilde \Omega=\widetilde\Omega_1\cup\widetilde\Omega_2$ we have that $\widetilde \Omega$ and the original set $\Omega$ have the same spectrum and the same measure. Thus $\widetilde\Omega$ is a solution of \eqref{introP}. Let $(u_1,\dots,u_l)$ and $(v_1,\dots,v_{k-l})$ be the vectors of the first eigenfunctions on $\widetilde\Omega_1$ and $\widetilde\Omega_2$. Let $U_0$ and $V_0$ be two limits of the blow-up sequences of these two vectors in zero. By the optimality and the homogeneity of $|U_0|$ and $|V_0|$, together with the fact that they are non-zero (see Proposition \ref{blowup}) we have that necessarily $\{|U_0|>0\}$ and $\{|V_0|>0\}$ are two complementary half-spaces. On the other hand there is a blow-up limit $W_0\in \mathcal{BU}_U(0)$ whose components are precisely the ones of $U_0$ and $V_0$. Now, by the optimality of $|W_0|$, it has to be a non-negative non-zero harmonic function on $B_1$ vanishing in zero, in contradiction with the maximum principle, so $\Omega$ is disconnected. 
\end{proof}

The proof  of Proposition \ref{blowup} is based on the fact that if $U$ is the vector eigenfunctions on the optimal set for $\lambda_1+\dots+\lambda_k$, then $U_{r,x_0}$ satisfies a quasi-minimality condition of the form \eqref{Uoptcond}. This is a direct consequence from the scaling properties of the functional $\mathcal F_\qm$ defined in Section \ref{sec:eigen}. Since it is essential for the proof of Proposition \ref{blowup}, we show it in the following Lemma. 
\begin{lemma}\label{scalingopt}
Suppose that $U\in H^1(\R^d;\R^k)\cap L^{\infty}(\R^d;\R^k)$ and that there are constants $\qm>0$ and $\eps>0$ such that $U$ satisfies the quasi-minimality condition~\ref{Uoptcond}.
Then, for every $x_0\in\R^d$, $U_{r,x_0}$ satisfies 
\begin{equation*}
\begin{split}
\int_{\R^d}|\nabla U_{r,x_0}|^2\,dx+\Lambda\big|\{|U_{r,x_0}|>0\}\big|\le \Big(1+\qm r^{d+1}\|U_{r,x_0}-\widetilde U\|_{L^1}\Big)\int_{\R^d}|\nabla \widetilde U|^2\,dx+\Lambda\big| \{|\widetilde U|>0\}\big|,\\
\quad\text{for every}\quad \widetilde U\in H^1(\R^d;\R^k)\quad\text{such that}\quad \|\widetilde U\|_{L^\infty}\le \frac{1}{\eps r}\quad\text{and}\quad \|U-\widetilde U\|_{L^1}\le \frac{\eps}{r^{d+1}}. 
\end{split}
\end{equation*}
\end{lemma}
\begin{proof}
Assume for simplicity that $x_0=0$. Let $\widetilde U\in H^1(\R^d;\R^k)\cap L^\infty(\R^d;\R^k)$ be such that 
$$\|U_{r}-\widetilde U\|_{L^1}\le \frac{\eps}{r^{d+1}}\qquad\text{and}\qquad \|\widetilde U\|_{L^\infty}\le \frac{1}{\eps r},$$
and consider the functions $\ds\Phi=U_{r}-\widetilde U$, $\Phi^r(x):=r\Phi\left(\frac{x}r\right)$ and $\tilde U^r(x):=r\tilde U\left(\frac{x}r\right)$. We notice that $$\|\Phi^r\|_{L^1}=r^{d+1}\|\Phi\|_{L^1}\le \eps\qquad\text{and}\qquad \|\tilde U^r\|_{L^\infty}=r\|\tilde U\|_{L^\infty}\le \frac1\eps,$$
and so we may use $\tilde U^r=U+\Phi^r$ to test the optimality of $U$:
\begin{align*}
\int_{\R^d}|\nabla U_{r}|^2\,dx+\Lambda\big|\{|U_{r}|>0\}\big|&=\frac1{r^d}\int_{\R^d}|\nabla U|^2\,dx+\frac\Lambda{r^d}\big|\{|U|>0\}\big|\\
&\le \big(1+\qm \|\Phi^r\|_{L^1}\big)\frac1{r^d}\int_{\R^d}|\nabla \tilde U^r|^2\,dx+\frac\Lambda{r^d}\big|\{|\tilde U^r|>0\}\big|\\
&=\Big(1+\qm r^{d+1}\|\Phi\|_{L^1}\Big)\int_{\R^d}|\nabla \widetilde U|^2\,dx+\Lambda\big| \{|\widetilde U|>0\}\big|,
\end{align*}
which gives the claim.
\end{proof}

\begin{prop}[Convergence of the blow-up sequences]\label{blowupexistence}
Let $U$ be a Lipschitz continuous local minimizer of $\mathcal F_\qm$ in the open set $\Dr\subset\R^d$. Suppose that $(r_n)_{n\in\N}$ and $(x_n)_{n\in\N}\subset\partial\{|U|>0\}$ are two sequences such that, for some $x_0\in\partial\{|U|>0\}$ and $U_0:\R^d\to\R^k$ Lipschitz continuous, we have
$$\lim_{n\to\infty}r_n=0\ ,\quad\lim_{n\to\infty}x_n=x_0\quad\text{and}\quad\lim_{n\to\infty}U_{r_n,x_n}=U_0,$$
where the convergence of $U_{r_n,x_n}$ is to be intended locally uniform in $\R^d$. Then, for every $R>0$, the following properties hold: 
\begin{enumerate}[(a)]
\item The sequence $\ds U_{r_n,x_n}(x):=\frac1{r_n}U(x_n+r_n x)$ converges to $U_0$ strongly in $H^1(B_R;\R^k)$. 
\item The sequence of characteristic functions $\ind_{\Omega_n}$ converges in $L^1(B_R)$ to $\ind_{\Omega_0}$, where 
$$\Omega_n:=\{|U_{r_n}|>0\}\quad\text{and}\quad \Omega_0:=\{|U_0|>0\}.$$
  
\item The sequences of closed sets $\overline\Omega_n$ and $\Omega_n^c$ converge Hausdorff in $B_R$ respectively to $\overline\Omega_0$ and $\Omega_0^c$.
\item $U_0$ is non-degenerate at zero, that is, there is a dimensional constant $c_d>0$ such that 
$$\|U_0\|_{L^\infty(B_r)}\ge c_d\,r\quad\text{for every}\quad r>0.$$

\end{enumerate}
\end{prop}
\begin{proof}
We set for simplicity $U_n=U_{r_n,x_n}$ and we divide the proof in some steps, for sake of clarity.\\
{\bf Step 1.} Since  $U_{n}$ is bounded in $H^1_{loc}(\R^d;\R^k)$ (being uniformly Lipschitz) we have that $U_{n}$ converges weakly in $H^1_{loc}$ to $U_0\in H^1_{loc}(\R^d;\R^k)$. By the definition of $\Omega_n$ and the fact that $|U_{n}|$ converges locally uniformly to $|U_0|$ we have that 
$$\ind_{\Omega_0}\le \liminf_{n\to\infty}\ind_{\Omega_n}.$$ 
\noindent {\bf Step 2.} Let us now prove that $U_{n}$ converges strongly in $H^1_{loc}(\R^d;\R^k)$ to $U_0$ and that $\ind_{\Omega_n}$ converges to $\ind_{\Omega_0}$ pointwise on $\R^d$. Fixed a ball $B_R\subset\R^d$ it is sufficient to prove that 
\begin{equation}\label{limsupineq}
\lim_{n\to\infty}\left(\int_{B_R}|\nabla U_{n}|^2\,dx+\Lambda|B_R\cap{\Omega_n}|\right)=\int_{B_R}|\nabla U_0|^2\,dx+\Lambda|B_R\cap\Omega_0|.
\end{equation}
We notice that the function $U_{n}$ is a local minimizer of 
$$\F_n(V)=\left(1+r_n^{d+1}\qm\|U_{n}-V\|_{L^1}\right)\int_{\R^d}|\nabla V|^2\,dx+\Lambda\big|\{|V|>0\}\big|.$$
Consider a function $\varphi\in C^\infty_c(\R^d)$ such that $0\le \varphi\le 1$ and $B_R=\{\varphi=1\}$. We introduce the test function
$$\tilde U_n=\varphi U_0+(1-\varphi)U_{n}.$$
The optimality of $U_{n}$ now gives 
\begin{equation}\label{optUn}
\begin{split}
\int_{\{\varphi>0\}}|\nabla U_{n}|^2\,dx&+\Lambda\big|\{|U_{n}|>0\}\cap \{\varphi>0\}\big|\\&\le \left(1+r_n^{d+1}\qm\|U_{n}-\tilde U_n\|_{L^1}\right)\int_{\{\varphi>0\}}|\nabla \tilde U_n|^2\,dx+\Lambda\big|\{|\tilde U_n|>0\}\cap \{\varphi>0\}\big|\\
&\le \left(1+r_n^{d+1}\qm\|\varphi(U_0-U_{n})\|_{L^1}\right)\int_{\{\varphi>0\}}|\nabla \tilde U_n|^2\,dx+\Lambda\big|\{|\tilde U_n|>0\}\cap \{\varphi>0\}\big|\\
&\le \left(1+r_n^{d+1}\qm\|\varphi(U_0-U_{n})\|_{L^1}\right)\int_{\{\varphi>0\}}|\nabla \tilde U_n|^2\,dx\\
&\qquad+\Lambda\Big(\big|\{\varphi=1\}\cap \{|U_0|>0\}\big|+\big|\{0<\varphi<1\}\big|\Big)\\
\end{split}
\end{equation}
Since $U_{n}$ converges strongly $L^2(B_R;\R^k)$ and weakly $H^1_{loc}(\R^d;\R^k)$ to $U_0$, we can estimate 
\begin{align*}
\int_{\{\varphi>0\}}&|\nabla U_{n}|^2-\int_{\{\varphi>0\}}|\nabla \tilde U_n|^2\,dx\\
&=\int_{\{\varphi>0\}}|\nabla U_{n}|^2-\int_{\{\varphi>0\}}|\nabla (\varphi U_0+(1-\varphi)U_{n})|^2\,dx\\
&=\int_{\{\varphi>0\}}(\nabla U_{n}-\nabla (\varphi U_0+(1-\varphi)U_{n}))\cdot (\nabla U_{n}+\nabla (\varphi U_0+(1-\varphi)U_{n}))\,dx\\
&=\int_{\{\varphi>0\}}(\varphi\nabla (U_{n}-U_0)+(U_{n}-U_0)\nabla\varphi)\cdot (\varphi\nabla (U_{n}+U_0)+(U_{n}+U_0)\nabla\varphi+ 2\nabla((1-\varphi)U_{n}))\,dx\\
&=\int_{\{\varphi>0\}}\varphi^2(|\nabla U_{n}|^2-|\nabla U_0|^2)\,dx+2\int_{\{\varphi>0\}}\varphi\nabla(U_{n}-U_0)\cdot (1-\varphi)\nabla U_{n}\,dx+o(1/n)\\
&=\int_{\{\varphi>0\}}(1-(1-\varphi)^2)(|\nabla U_{n}|^2-|\nabla U_0|^2)\,dx+o(1/n).
\end{align*}
Now since $|\nabla U_{n}|$ converges weakly in $L^2(\{0<\varphi<1\};\R)$ to $|\nabla U_0|$, we have that 
$$\limsup_{n\to\infty}\int_{\{\varphi>0\}}\left(|\nabla U_{n}|^2-|\nabla \tilde U_n|^2\right)\,dx\ge \limsup_{n\to\infty}\int_{\{\varphi=1\}} \left(|\nabla U_{n}|^2-|\nabla U_0|^2\right)\,dx.$$
Substituting in the inequality~\eqref{optUn} above we obtain 
\begin{align*}
\limsup_{n\to\infty}&\left(\int_{\{\varphi=1\}} \left(|\nabla U_{n}|^2-|\nabla U_0|^2\right)\,dx+\Lambda(|\{\varphi=1\}\cap \Omega_n|-|\{\varphi=1\}\cap \Omega_0|)\right)\\
&\le \limsup_{n\to\infty}\left(\int_{\{\varphi>0\}}\left(|\nabla U_{n}|^2-|\nabla \tilde U_n|^2\right)\,dx+\Lambda(|\{\varphi=1\}\cap \Omega_n|-|\{\varphi=1\}\cap \Omega_0|)\right)\\
&\qquad\qquad\le \Lambda|\{0<\varphi<1\}|.
\end{align*}
Now, since $\varphi$ is arbitrary outside $B_R$, we get \eqref{limsupineq}. So we have proved part $(a)$ and $(b)$ of the Proposition.\\
\noindent {\bf Step 3. } It is well-known that the convergence $L^1$ of the sequence of characteristic functions $\ind_{\Omega_n}$ together with the fact that each $\Omega_n$ satisfies the density estimate 
$$\eps_0|B_r|\le |\Omega_n\cap B_r|\le (1-\eps_0)|B_r|,\qquad\forall r<r_0/r_n,$$
gives that both $\overline\Omega_n$ and $\Omega_n^c$ converge Hausdorff respectively to $\overline\Omega_0$ and $\Omega_0^c$ locally in $\R^d$, hence also part $(c)$ of the statement is concluded.\\
{\bf Step 4.} It remains only to prove the non-degeneracy of $U_0$.
We first note that every function $U_{r_n}$ is non-degenerate in the following sense: 
\begin{equation}\label{nondegUrn}
y\in \overline\Omega_n\Rightarrow \|U_{n}\|_{L^\infty(B_r(y))}\ge c_0 r,\ \forall r\le r_0/r_n.
\end{equation}
In fact if $y\in \overline\Omega_n$, then $r_n y\in \overline\Omega=\overline{\{|U|>0\}}$. By the non-degeneracy of $U$ we obtain 
$$r_n\|U_{n}\|_{L^\infty(B_r(y))}=\|U\|_{L^\infty(B_{r r_n}(x_n+r_n y))}\ge c_0 r r_n ,\ \forall r\le r_0/r_n,$$
which is precisely \eqref{nondegUrn}.
Our claim that the function $U_0$ is non-degenerate means 
\begin{equation}\label{nondegUinfty}
y\in \overline\Omega_0\Rightarrow \|U_0\|_{L^\infty(B_r(y))}\ge \frac{c_0}4 r,\ \forall r>0.
\end{equation}
Suppose that $y\in\overline\Omega_0$ and $r>0$. Then there is $y'\in B_{r_2}(y)$ such that $|U_0|(y')>0$. Then for $n$ large enough $y'\in \overline\Omega_n$. By the non-degeneracy of $U_{n}$ we have that there is a point $y_n\in \overline{B_{r/2}(y')}$ such that
$$2|U_{n}|(y_n)\ge \|U_{n}\|_{L^\infty(B_{r/2}(y'))}\ge c_0 r/2.$$
We can assume that $y_n$ converges to some $y_\infty\in  \overline{B_{r/2}(y')}$, for which the uniform convergence of $U_{n}$ gives $|U_0|(y_\infty)\ge c_0 r/4,$
and so we have \eqref{nondegUinfty}.
\end{proof}

\begin{lemma}[Optimality of the blow-up limits]\label{Uinftymin}
Let $U\in H^1(\R^d;\R^k)$ be a Lipschitz continuous function satisfying the quasi-minimality condition \eqref{Uoptcond}. Let $x_0\in\partial\{|U|>0\}$ and $U_0\in \mathcal{BU}_U(x_0)$. 
Then $ U_0$ is a local minimizer of the functional $\mathcal{F}_0$. 
\end{lemma}
\begin{proof}
Let $x_0=0$ and $B_R\subset\R^d$ be a fixed ball. We first notice that if $U$ satisfies \eqref{Uoptcond} and $r>0$, then $U_r(x)=\frac1rU(rx)$ satisfies the following quasi-minimality condition in the ball $B_R$ (see Lemma \ref{scalingopt})
\begin{equation}\label{Uroptcond}
\begin{array}{ll}
\ds\big(1+\qm r^{d+1}\|U_r-\widetilde U\|_{L^1}\big)\int_{B_{R}}|\nabla U_r|^2\,dx+\Lambda\big|\{|U_r|>0\}\cap B_{R}\big|\\
\ds\qquad\qquad\qquad\qquad\le \big(1+\qm r^{d+1}\|U_r-\widetilde U\|_{L^1}\big)\int_{B_{R}}|\nabla \widetilde U|^2\,dx+\Lambda\big|\{|\widetilde U|>0\}\cap B_{R}\big|\\
\ds\qquad\qquad\qquad\qquad\qquad\qquad +\qm r\|U_r-\widetilde U\|_{L^1}\int_{\R^d}|\nabla U|^2\,dx,
\end{array}
\end{equation}
for every $\widetilde U\in H^1(\R^d;\R^k)\cap L^\infty(\R^d;\R^k)$ such that $U_r-\widetilde U\in H^1_0(B_R,\R^k)$ and 
$$\|U_r-\widetilde U\|_{L^1}\le \frac{\eps}{r^{d+1}}\qquad\text{and}\qquad \|\widetilde U\|_{L^\infty}\le \frac{1}{\eps r}.$$

 Let now $\widetilde U\in H^1_{loc}(\R^d;\R^k)\cap L^\infty_{loc}(\R^d;\R^k)$ be such that $U_0-\widetilde U\in H^1_0(B_R,\R^k)$ and let $\eta\in C^\infty_c(B_R)$ be such that $0\le \eta\le 1$. We consider a sequence $U_{r_n}$ converging to $U_0$ is sense of Proposition \ref{blowupexistence}.
We recall that $U_{r_n}\to U_0$ both uniformly in $B_R$ and strongly in $H^1(B_R)$.
Consider the test function $W_n=\widetilde U+(1-\eta)(U_{r_n}-U_0)$. Since $\widetilde U=U_0$ outside $B_R$ we have that $W_n=U_{r_n}$ outside $B_R$. 
Moreover, since $W_n-U_{r_n}=\widetilde U-U_0-\eta (U_{r_n}-U_0)$ and $U_{r_n}\to U_0$ in $L^1(B_R)$ we have that, for $n\ge n_0$ (where $n_0$ does not depend on $\eta$ but only on the sequence $r_n$), 
$$\|W_n-U_{r_n}\|_{L^1}\le 2\|\widetilde U-U_0\|_{L^1}\qquad\text{and}\qquad\|W_n-U_{r_n}\|_{L^\infty}\le 2\|\widetilde U-U_0\|_{L^\infty},$$ 
and so $W_n$ can be used as a test function in \eqref{Uroptcond}, thus obtaining 
\begin{equation*}
\begin{array}{ll}
\ds\big(1+\qm r_n^{d+1}\|U_{r_n}-W_n\|_{L^1}\big)\int_{B_{R}}\Big(|\nabla U_{r_n}|^2-|\nabla W_n|^2\Big)\,dx+\Lambda\big|\{|U_{r_n}|>0\}\cap B_{R}\big|\\
\ds\qquad\qquad\qquad\qquad\le \Lambda\Big(\big|\{|\widetilde U|>0\}\cap\{\eta=1\}\big| +\big|\{0<\eta<1\}\big|\Big)\\
\ds\qquad\qquad\qquad\qquad\qquad\qquad+2\qm r_n\|\widetilde U-U_0\|_{L^1}\int_{\R^d}|\nabla U|^2\,dx.
\end{array}
\end{equation*}
Now since $U_{r_n}\to U_0$ in $H^1(B_R;\R^k)$ and $W_n\to \tilde U$ in $H^1(B_R;\R^k)$ we have 
\begin{equation*}
\begin{array}{ll}
\ds\int_{B_{R}}|\nabla U_0|^2\,dx+\Lambda\big|\{|U_0|>0\}\cap B_{R}\big|\\
\ds\qquad\qquad\qquad\qquad\le \int_{B_{R}}|\nabla \widetilde U|^2\,dx+\Lambda\Big(\big|\{|\widetilde U|>0\}\cap\{\eta=1\}\big| +\big|\{0<\eta<1\}\big|\Big).
\end{array}
\end{equation*}
Since we can choose $\eta$ such that $|\{\eta=1\}|$ is arbitrarily close to $|B_R|$ we
obtain
\begin{equation*}
\begin{array}{ll}
\ds\int_{B_{R}}|\nabla U_0|^2\,dx+\Lambda\big|\{|U_0|>0\}\cap B_{R}\big|\le \int_{B_{R}}|\nabla \widetilde U|^2\,dx+\Lambda\big|\{|\widetilde U|>0\}\cap B_R\big|.
\end{array}
\end{equation*}
\end{proof}


\begin{lemma}[Homogeneity of the blow-up limits]\label{Uinftyhom}
Let $U\in H^1(\R^d;\R^k)$ be a Lipschitz continuous function satisfying the quasi-minimality condition \eqref{Uoptcond}. Let $x_0\in\partial\{|U|>0\}$ and $U_0\in \mathcal{BU}_U(x_0)$. Then $U_0$ is a one-homogeneous function.
\end{lemma}
\begin{proof}
Let the sequence $r_n\to 0$ be such that the sequence $U_n(x):=\frac1{r_n}U(x_0+r_nx)$ converges to $U_0$ both uniformly and (see Proposition \ref{blowupexistence}) strongly in $H^1(B_R;\R^k)$, for every ball $B_R\subset\R^d$. Let $\phi_n$ be the Weiss functional corresponding to $U_n$
\begin{equation}\label{phin}
\phi_n(r):=\phi(U_n,0,r)=\frac{1}{r^{d}}\int_{B_r}|\nabla U_{n}|^2\,dx-\frac{1}{r^{d+1}}\int_{\partial B_r}|U_n|^2\,d\HH^{d-1}+\frac{\Lambda}{r^{d}}\big|\{|U_n|>0\}\cap B_r\big|.
\end{equation}  
We notice that  
\begin{equation}\label{phiscal}
\phi_n(r)=\phi(U,x_0,r_n r)\quad\text{for every}\quad r>0,
\end{equation}
where $\phi(U,x_0,r)$ is the Weiss functional corresponding to $U$ from \eqref{weiss_fun}. 
By \eqref{phiscal} and the fact that the limit $\ds\lim_{r\to0}\phi(U,x_0,r)$ exists (see Proposition \ref{mono_weiss}) we have that for every fixed $r>0$
\begin{equation}\label{bwproplimit}
\lim_{n\to\infty}\phi_n(r)=\lim_{n\to\infty}\phi(U,x_0,r_n r)=\lim_{\rho\to0}\phi(U,x_0,\rho).
\end{equation}
On the other hand Proposition \ref{blowupexistence} gives that 
$$\lim_{n\to\infty}\phi_n(r)=\phi_0(U_0,0,r),$$
Now since $\phi_0(U_0,0,r)$ is constant in $r$ (due to \eqref{bwproplimit}) and $U_0$ is optimal (due to Proposition \ref{Uinftymin}) we can apply Proposition \ref{mono_weiss_global} and finally obtain that $U_0$ is one-homogeneous function on $\R^d$.
\end{proof}

\begin{oss}\label{remsphere}
In the following Lemma and in Section~\ref{sec:regularity} we will use some rather well known facts about eigenvalues of the spherical Laplacian $\Delta_S$ on regions of the sphere. For more details we refer to~\cite{sperner,fh}, but we summarize here the main facts that we need in the following.
\begin{itemize}
\item Let $S\subset\partial B_1$ be an open subset of the sphere $\partial B_1\subset\R^d$, for $d\ge 2$, and let ${\bf C}_S=\{r\theta\, :\, \theta\in S,\ r>0\}$ be the cone generated by $S$. Then, given an $\alpha$-homogeneous function $u:{\bf C}_S\to \R$
for some $\alpha>0$, we have that $u$ is a solution of the problem
$$\Delta u=0\quad\text{in}\quad {\bf C}_S,\qquad u=0\quad\text{on}\quad \partial {\bf C}_S,$$
if and only if the trace $\varphi=u|_{\partial B_1}$ is a solution of the problem 
$$-\Delta_{S} \varphi =\lambda \varphi\quad\text{in}\quad S,\qquad \varphi=0\quad\text{on}\quad \partial S,$$
where $\lambda=\alpha(\alpha+d-2)$ and $\Delta_S$ denotes the Laplace-Beltrami operator on the sphere $\partial B_1$. We denote by $\{\lambda_j(S)\}_{j\ge 1}$ the non-decreasing sequence of eigenvalues on set $S\subset\partial B_1$ counted with the due multiplicity. 
\item For the spherical sets $S$ we have the inequality 
\begin{equation}\label{sphericalFK}
\lambda_1(S)\ge d-1\quad\text{for every}\quad S\subset \partial B_1\quad\text{such that}\quad \HH^{d-1}(S)\le \frac{d\omega_d}{2},
\end{equation}
and the equality is achieved if and only if, up to a rotation, $S$ is the half-sphere 
$$\partial B_1^+=\left\{x=(x_1,\dots,x_d)\in\partial B_1\, :\, x_d>0\right\}.$$
\item As a consequence of \eqref{sphericalFK} we get that 
\begin{equation}\label{lambda2S}
\lambda_2(S)\ge d-1\quad\text{for every}\quad S\subset \partial B_1,
\end{equation}
where the equality is achieved if and only if, up to a rotation, $\partial B_1\cap\{x_d\neq0\}\subset S$. Indeed, if the second eigenfunction $\varphi_2\in H^1_0(S)$ changes sign, then we can apply \eqref{sphericalFK} to the sets $\{\varphi_2>0\}$ and $\{\varphi_2<0\}$. If $\varphi_2\ge 0$ on $S$, then the sets $\{\varphi_1>0\}$ ($\varphi_1\ge 0$ being the first eigenfunction on $S$) and $\{\varphi_2>0\}$ are disjoint and again the claim follows by \eqref{sphericalFK}.
\item As a consequence of \eqref{sphericalFK} and \eqref{lambda2S} we obtain that 
if $S\subset\partial B_1$ is such that $\lambda_1(S)\le d-1$ and $\HH^{d-1}(S)<d\omega_d$, then the first eigenvalue $\lambda_1(S)$ is simple, that is there exists a unique (non-negative) function $\varphi_1\in H^1_0(S)$ such that 
$$-\Delta_{S} \varphi_1 =\lambda_1(S)\varphi_1\quad\text{in}\quad S,\qquad \varphi_1=0\quad\text{on}\quad \partial S,\qquad \int_S\varphi_1^2=1.$$
\end{itemize}
\end{oss}

\begin{lemma}\label{Uinftyhom2}
Let $U\in H^1(\R^d;\R^k)$ be a Lipschitz continuous function satisfying the quasi-minimality condition \eqref{Uoptcond}. Let $x_0\in\partial\{|U|>0\}$ and $U_0\in \mathcal{BU}_U(x_0)$. Then, there is a unit vector $\xi\in \partial B_1\subset\R^k$ such that $U_0=\xi|U_0|$. 
\end{lemma}
\begin{proof}
By Lemma \ref{Uinftyhom} $U_0=(u_1,\dots,u_k)$ is a one-homogeneous function and so is $|U_0|$. Let $S:=\partial B_1\cap \{|U_0|>0\}$. We first notice that all the components $u_1,\dots,u_k$ of $U_0$ are harmonic functions on the cone $\{|U_0|>0\}=\{r\xi\ :\ \xi\in S,\ r>0\}$. Thus in polar coordinates we have that $u_i(r,\theta)=r\varphi_i(\theta)$, where $\varphi_i$ satisfies
$$-\Delta_{S}\varphi_i=(d-1)\varphi_i\quad\text{in}\quad S,\qquad \varphi_i=0\quad\text{on}\quad \partial S,$$
that is, $d-1$ is an eigenvalue of the spherical Laplacian $\Delta_S$ on $S$ and the non-zero components of $U_0$ are (non-normalized) eigenfunctions. Now since $|S|<|\partial B_1|$ ( due to the optimality of $U_0$ ) the last point of Remark \ref{remsphere} implies that the first eigenvalue $\lambda_1(S)$ is simple. 
Then, denoting  by $\varphi$ the first normalized eigenfunction on $S$, we get that there are constants $a_1,\dots,a_k$ such that $\varphi_i=a_i\varphi$, for every $i=1,\dots,k$. Setting $A=(a_1,\dots,a_k)$ we have that $|U_0|=|A|\varphi$. Since $U_0$ is not constantly zero on $\partial B_1$ (see Proposition \ref{blowupexistence}), we have that $|A|\neq0$ and thus, taking $\xi=|A|^{-1}A$ we have the claim.
\end{proof}

\begin{lemma}\label{Uinftyopt2}
Let $U\in H^1(\R^d;\R^k)$ be a Lipschitz continuous function satisfying the quasi-minimality condition \eqref{Uoptcond}. Let $x_0\in\partial\{|U|>0\}$ and $U_0\in \mathcal{BU}_U(x_0)$. Then, the scalar function $|U_0|$ is a local minimizer of the Alt-Caffarelli functional. 
\end{lemma}
\begin{proof}
We set for simplicity $u=|U_0|$. Let $\xi$ be the unit vector from Lemma \ref{Uinftyhom2}. Let $\tilde u\in H^1_{loc}(\R^d)$ be such that the difference $u-\tilde u$ is supported in the ball $B_R$. Then the same holds for the function $U_0-\xi\tilde u$. By the optimality of $U_0$ we have 
\begin{align*}
\int_{B_R}|\nabla u|^2\,dx+\Lambda|\{u>0\}\cap B_R|&=\int_{B_R}|\nabla U_0|^2\,dx+\Lambda|\{|U_0|>0\}\cap B_R|\\
&\le\int_{B_R}|\nabla (\xi \tilde u)|^2\,dx+\Lambda|\{|\xi\tilde u|>0\}\cap B_R|\\
&=\int_{B_R}|\nabla \tilde u|^2\,dx+\Lambda|\{|\tilde u|>0\}\cap B_R|,
\end{align*}
which proves the claim.
\end{proof}

\section{Regularity of the free boundary}\label{sec:regularity}
In this section we conclude the proof of Theorem \ref{main}. 

\subsection{The optimality condition on the free boundary}
It is well-known (see for example \cite{altcaf}) that if $u$ is a local minimizer of the Alt-Caffarelli functional 
$$H^1_{loc}(\R^d)\ni u\mapsto\mathcal E_0(u):=\int|\nabla u|^2\,dx+\Lambda|\{u>0\}|,$$
and the boundary $\partial\{u>0\}$ is smooth, then the following boundary optimality condition holds :
$$|\nabla u|=\sqrt{\Lambda}\quad\text{on}\quad\partial \{u>0\}. $$
There are various ways to state this optimality for free boundaries that are not a priori smooth (see for example \cite{altcaf} and \cite{desilva}). In the case of vector-valued functionals the most appropriate one seems to be the approach exploiting the notion of a viscosity solution. 

\begin{deff}\label{viscoptimality}
Let $\Omega\subset\R^d$ be an open set and ${\bf\lambda}=(\lambda_1,\dots,\lambda_k)\in\R^k$ a vector with positive coordinates. We say that the continuous function $U=(u_1,\dots,u_k):\overline\Omega\to\R^k$ is a viscosity solution of the problem 
$$-\Delta U={\bf\lambda} U\quad\text{in}\quad\Omega,\qquad U=0\quad\text{on}\quad\partial\Omega,\qquad |\nabla |U||=\sqrt \Lambda\quad\text{on}\quad\partial\Omega,$$
if for every $i=1,\dots,k$ the component $u_i$ is a solution of the PDE
$$-\Delta u_i=\lambda_i u_i\quad\text{in}\quad\Omega,\qquad u_i=0\quad\text{on}\quad\partial\Omega,$$
and the boundary condition
$$|\nabla |U||=\sqrt \Lambda\quad\text{on}\quad\partial\Omega,$$
holds in viscosity sense, that is
\begin{itemize}
\item for every continuous function $\varphi:\R^d\to\R$ differentiable in $x_0\in\partial\Omega$ and such that ``$\varphi$ touches $|U|$ from below in $x_0$'' (that is $|U|-\varphi:\overline\Omega\to\R$ has a local minimum equal to zero in $x_0$), we have $|\nabla \varphi|(x_0)\le\sqrt\Lambda$.
\item for every continuous function $\varphi:\R^d\to\R$ differentiable in $x_0\in\partial\Omega$ and such that ``$\varphi$ touches $|U|$ from above in $x_0$'' (that is $|U|-\varphi :\overline\Omega\to\R$ has a local maximum equal to zero in $x_0$), we have $|\nabla \varphi|(x_0)\ge\sqrt\Lambda$.
\end{itemize}
\end{deff}

\begin{lemma}\label{optvisc}
Let $\Omega$ be a solution of the problem \eqref{introP}, $U=(u_1,\dots,u_k)$ be the vector of the first $k$ eigenfunctions on $\Omega$, $\lambda=(\lambda_1(\Omega),\dots,\lambda_k(\Omega))$ and $\ds\Lambda=\frac2d\big(\lambda_1(\Omega)+\dots+\lambda_k(\Omega)\big)$. Then $U$ is a viscosity solution to the problem 
\begin{equation}\label{viscsolUtion}
-\Delta U=\lambda U\quad\text{in}\quad\Omega,\qquad U=0\quad\text{on}\quad\partial\Omega,\qquad  |\nabla |U||=\sqrt{\Lambda}\quad\text{on}\quad\partial\Omega.
\end{equation}
\end{lemma}
\begin{proof}
From Theorem~\ref{knownstuffmainopt0} it follows that $|U|:\R^d\to\R^k$ is Lipschitz continuous. We only have to prove that the identity 
$|\nabla |U||=\sqrt \Lambda$ 
holds in viscosity sense on the boundary $\partial\Omega$. 

{\it Step 1. }Suppose first that $\varphi$ touches $|U|$ from below in $x_0\in\partial\Omega$ and assume $x_0=0$. Consider the blow-up sequences 
\begin{equation*}
U_n(x)=\frac1{r_n}U(r_nx)\qquad\text{and}\qquad\varphi_n(x)=\frac1{r_n}\varphi(r_nx),
\end{equation*}
for a sequence of radii $r_n\to 0$. Up to a subsequence we have that the blow-up limits 
\begin{equation}\label{bulim1}
U_0=\lim_{n\to\infty}U_n(x)\qquad\text{and}\qquad\varphi_0=\lim_{n\to\infty}\varphi_n(x),
\end{equation}
exist where the convergence is locally uniform in $\R^d$.  We first notice that, as $\varphi$ is smooth, we have $\varphi_0(x)=\xi\cdot x$ for a vector $\xi\in\R^d$. Without loss of generality we may assume that $\xi=a e_d$ for some constant $a>0$, thus 
\begin{equation}\label{bulim112}
|\nabla\varphi(0)|=|\nabla\varphi_0(0)|=a\qquad\text{and}\qquad \varphi_0(x)=a x_d.
\end{equation}
Now, since $|U_0|\ge \varphi_0$, we obtain that $|U_0|>0$ on $\{x_d>0\}$. By Proposition \ref{Uinftyhom} we have that $U_0$ is a $1$-homogeneous harmonic function on the cone $\{|U_0|>0\}\supset\{x_d>0\}$. Thus, necessarily $U_0=0$ on the hyperplane $\{x\in\R^d\ :\ x_d=0\}$ and by the second point of Remark \ref{remsphere} we have only two possibilities:
$$\{|U_0|>0\}=\{x_d>0\}\qquad\text{or}\qquad\{|U_0|>0\}=\{x_d\neq0\}.$$
The second case is ruled out since, due to Proposition \ref{blowup}, $|U_0|$ is a local minimizer of the Alt-Caffarelli functional and so it has to satisfy an exterior density estimate, which is not the case of the set $\{x_d\neq0\}$. Thus the only possibility is $\{|U_0|>0\}=\{x_d>0\}$. In particular the boundary $\partial\{|U_0|>0\}$ is smooth as well as the function $U_0$ whose components are linear functions. Since $|U_0|$ is a minimizer of the Alt-Caffarelli functional, it satisfies the optimality condition 
\begin{equation}\label{altcafoptcond}
|\nabla |U_0||=\sqrt{\Lambda}\quad\text{on}\quad \{x_d=0\}.
\end{equation}
Thus we obtain that $|U_0|=\sqrt\Lambda x_d^+$. Now, by the inequality $|U_0|\ge \varphi_0$, we get that $a\le \sqrt\Lambda$, which concludes the proof of Step 1.

{\it Step 2. }Suppose now that $\varphi$ touches $|U|$ from above at $x_0=0$ and once again we consider the blow-up limits $U_0$ and $\varphi_0$ defined in \eqref{bulim1} and we assume that $\varphi_0$ is as in \eqref{bulim112}. Due to the non-degeneracy of $U_0$ (see Proposition \ref{blowupexistence}) we get that $U_0\not\equiv0$ and $a>0$.  Since $U_0\le \varphi_0$ we have that the cone $\{|U_0|>0\}$ is contained in the half-space $\{x_d>0\}$. By the $1$-homogeneity of $U_0$ and Remark \ref{remsphere} we obtain that necessarily $\{|U_0|>0\}=\{x_d>0\}$. In particular, $\partial\{|U_0|>0\}$ is smooth and $|U_0|$ is linear. In conclusion, applying as above Proposition \ref{blowup}, we get that $|U_0|$ satisfies \eqref{altcafoptcond}, which gives that $|U_0|=\sqrt\Lambda x_d^+$ and $a\ge \sqrt\Lambda$.
\end{proof}

\subsection{Regular and singular parts of the free boundary}
Let $\Omega$ be a solution of \eqref{introP}. We define the regular part of the free boundary (or the regular set) $Reg(\partial\Omega)$ to be the set 
of points of density $1/2$ of $\Omega$, that is, $Reg(\partial\Omega):=\Omega^{(1/2)}.$
On the other hand, the singular part of the free boundary (or the singular set) $Sing(\partial\Omega)$ is defined as the complementary of $Reg(\partial\Omega)$
$$Sing(\partial\Omega):= \partial\Omega\setminus Reg(\partial\Omega).$$
In this subsection we prove that $Reg(\partial\Omega)$ is relatively open in $\partial\Omega$ (i.e. $Sing(\partial\Omega)$ is a closed set).

\begin{lemma}[Density gap]\label{gap}
There exists a constant $\delta>0$ such that for every non-trivial $1$-homogeneous local minimizer $u$ of the Alt-Caffarelli functional
$$H^1_{loc}(\R^d)\ni u\mapsto \mathcal E_0(u)=\int|\nabla u|^2\,dx+\Lambda|\{u>0\}|,$$ 
we have that 
$$0\notin\Omega_u^{(\gamma)},\quad\text{for every}\quad\gamma\in(1/2,1/2+\delta),$$
where $\Omega_u=\{u>0\}$.
\end{lemma}
\begin{proof}
Suppose by contradiction that there are an infinitesimal sequence of positive real numbers $\delta_n$ and a sequence $u_n$ of $1$-homogeneous non-zero local minimizers of $\mathcal E_0$ such that 
$$\frac{|B_r\cap\Omega_n|}{|B_r|}=\frac12+\delta_n,\quad\text{for every}\quad r>0,$$
where $\Omega_n=\{u_n>0\}$. By \cite[Section 3]{altcaf} the sequence $u_n$ is uniformly Lipschitz and non-degenerate and so, up to a subsequence it converges to a $1$-homogeneous non-zero function $u_0$. Reasoning as in \cite[Lemma 5.4]{altcaf} it is straightforward to check that $u_0$ is a local minimizer of  $\mathcal E_0$ and, in particular, harmonic on the cone $\Omega_0=\{u_0>0\}$. Moreover, using the density assumption on $\Omega_n$ and passing to the limit as $n\to\infty$ we deduce
$$\frac{|B_r\cap\Omega_0|}{|B_r|}\le\frac12,\quad\text{for every}\quad r>0.$$
Thus, by the second point of Remark \ref{remsphere}, up to a change of coordinates we may assume, that $\Omega_0=\{x_d>0\}$ and $u_0(x)=a x_d^+$, for some $a>0$. By the uniform convergence of $u_n$, for every $\eps>0$ we can find $n_0$ such that 
$$a(x_d-\eps)_+\le u_n(x)\le a(x_d+\eps)_+\quad\text{for every}\quad x\in B_1,\quad n\ge n_0.$$ 
Applying Theorem 1.1 from \cite{desilva} we obtain that for $n$ large enough $\partial\Omega_n$ is $C^{1,\alpha}$ and so $0\in\Omega_n^{(1/2)}$. In particular $\delta_n=0$ in contradiction with the initial assumption.
\end{proof}

\begin{lemma}\label{gap42}
Let $\Omega$ be a solution of \eqref{introP} and $U=(u_1,\dots,u_k)$ be the vector of the first $k$ eigenfunctions on $\Omega$. Then the following facts do hold:
\begin{enumerate}[(i)]
\item For every boundary point $x_0\in\partial\Omega$ we have that
$$\liminf_{r\to0}\frac{|B_r(x_0)\cap\Omega|}{|B_r|}\ge \frac12\ .$$
\item For every $\gamma\ge 1/2$ we have 
$$\Omega^{(\gamma)}=\Big\{x_0\in\partial\Omega\ :\ \lim_{r\to0}\phi(U,x_0,r)=\Lambda\omega_d \gamma\Big\},$$
where we recall that $\omega_d=|B_1|$ and $\phi(U,x_0,r)$ is the Weiss functional defined in \eqref{weiss_fun}.
\item There is a constant $\delta>0$ such that 
$$\ds\partial\Omega=\bigcup_{\gamma\in\left\{\frac12\right\}\cup\left[\frac12+\delta,1\right[}\Omega^{(\gamma)}.$$
\end{enumerate}
\end{lemma}
\begin{proof}
\begin{enumerate}[(i)]
\item Suppose that this is not the case. Then, there is a point $x_0=0$ and a sequence $r_n\to0$ such that 
$$\lim_{n\to\infty}\frac{|B_{r_n}\cap\Omega|}{|B_{r_n}|}<\frac12.$$
Setting $U_n(x)=\frac1{r_n}U(r_nx)$ and $\Omega_n=\{|U_n|>0\}$ we can suppose that $U_n$ converges in $H^1_{loc}(\R^d;\R^k)$ to a non-zero $1$-homogeneous function $U_0$, such that $|U_0|$ is a one-homogeneous local minimizer of the Alt-Caffarelli functional $\mathcal E_0$. Moreover, we can suppose that the sequence of conic level sets $\Omega_n$ converges in $L^1_{loc}$ to the cone $\Omega_0=\{|U_0|>0\}$. In particular we have 
$$\frac{|B_{1}\cap\Omega_0|}{|B_{1}|}=\lim_{n\to\infty}\frac{|B_{1}\cap\Omega_n|}{|B_{1}|}=\lim_{n\to\infty}\frac{|B_{r_n}\cap\Omega|}{|B_{r_n}|}<\frac12,$$
which is a contradiction since there cannot be a non-trivial $1$-homogeneous harmonic function on a cone of density less that $1/2$.
\item Let $x_0\in\partial\Omega$. We suppose that $x_0=0$ and set $\phi(r):=\phi(U,x_0,r)$. By Proposition \ref{mono_weiss}, the limit $\ds\lim_{r\to0}\phi(r)$ does exist. We set $\gamma$ to be the limit 
$$\gamma:=\frac{1}{\Lambda\omega_d}\lim_{r\to0}\phi(r).$$
On the other hand, consider an arbitrary sequence $r_n\to 0$. There is a subsequence, that we still denote by $r_n$, such that the corresponding blow-up sequence $U_n(x):=\frac1{r_n}U(r_nx)$ converges locally uniformly in $\R^d$. 
Defining $\phi_n(r):=\phi(U_n,0,r)$ as in \eqref{phin} we have $\phi_n(r)=\phi(rr_n)$ and thus, as in Proposition \ref{Uinftyhom}, 
\begin{equation}\label{inftylimtorino}
\begin{array}{ll}
\ds\gamma&\ds=\frac{1}{\Lambda\omega_d}\lim_{n\to\infty}\phi(r r_n)=\frac{1}{\Lambda\omega_d}\lim_{n\to\infty}\phi_n(r)\\
&\ds=\frac{1}{\Lambda\omega_d}\left[\frac1{r^{d}}\left(\int_{B_r}|\nabla U_0|^2\,dx+\Lambda|\{|U_0|>0\}\cap B_r|\right)-\frac{1}{r^{d+1}}\int_{\partial B_r}|U_0|^2\,d\HH^{d-1}\right],
\end{array}
\end{equation}
where $U_0$ is the blow-up limit of $U_n$. By the $1$-homogeneity of $U_0$ and the fact that it is harmonic on $\{|U_0|>0\}$ we obtain that 
$$\frac{1}{r^{d}}\int_{B_r}|\nabla U_0|^2\,dx-\frac{1}{r^{d+1}}\int_{\partial B_r}|U_0|^2\,d\HH^{d-1}=0.$$
Thus, by\eqref{inftylimtorino}, Proposition \ref{blowupexistence} (2) and the fact that $\{|U_n|>0\}=r_n\Omega$, we get that 
$$\gamma=\frac{|\{|U_0|>0\}\cap B_r|}{|B_r|}=\lim_{n\to\infty}\frac{|\Omega_n\cap B_r|}{|B_r|}=\lim_{n\to\infty}\frac{|\Omega\cap B_{rr_n}|}{|B_{rr_n}|}.$$
Since the sequence $r_n$ is arbitrary we have that $x_0\in\Omega^{(\gamma)}$, which gives the claim. 
\item By the previous point, for every $x_0\in \partial\Omega$ the limit 
$$\frac{1}{\Lambda\omega_d}\lim_{r\to0}\phi(U,x_0,r),$$
exists and coincides with the density of $\Omega$ in $x_0$. By point {\it (i)} we have that $\gamma\ge 1/2$. On the other hand, by Lemma \ref{gap} we have that $\gamma>1/2+\delta$, which gives the claim.
\end{enumerate}
\end{proof}

\begin{oss}
We highlight that the claim of Lemma~\ref{gap42} $(ii)$ can be restated as follows:
\begin{equation}\label{density=phi}
\lim_{r\rightarrow 0}\frac{|\Om\cap B_r(x_0)|}{|B_r|}=\frac{1}{\Lambda\omega_d}\lim_{r\rightarrow 0}{\phi(U,x_0,r)},\quad\text{for every}\quad x_0\in\partial\Omega.
\end{equation}
\end{oss}
In the next Proposition we show that the regular part of the free boundary is relatively open in the topological boundary of an optimal set. This is due to a general principle which can be stated as follows:\\
{\it Suppose that $Y\subset X$ is a set for which there exists a function $f_Y:X\times[0,+\infty)\to[0,+\infty)$ such that:
\begin{itemize}
\item the function $f_Y(\cdot,r):X\to[0,+\infty)$ is continuous for every fixed $r>0$;
\item  the function $f_Y(x,\cdot):[0,+\infty)\to[0,+\infty)$ is continuous and non-decreasing for every fixed $x\in X$;
\item $Y=\{x\ :\ f_Y(x,0)=0\}$ and there is $\delta>0$ such that $\{x\ :\ 0<f_Y(x,0)<\delta\}=\emptyset$.
\end{itemize}
Then $Y$ is relatively open in $X$.}\\
In fact the first two points imply that the function $f_Y(\cdot,0):X\to[0,+\infty)$ is upper semi-continuous and this, combined with the last point, gives the conclusion. In our case the situation is slightly different but follows by the same principle. For sake of completeness we give here an elementary proof in our situation. 
\begin{prop}\label{gapprop}
Let $\Omega$ be a solution of \eqref{introP}. 
Then the regular set $Reg(\partial\Omega)$ is an open subset of $\partial\Omega$.
\end{prop}
\begin{proof}
Let $x_0\in Reg(\partial\Omega)=\Omega^{(1/2)}$. Suppose that there is a sequence $x_n\in Sing(\partial\Omega)=\partial\Omega\setminus \Omega^{(1/2)}$ such that $x_n\to x_0$. Let $U$ be the vector of the first $k$ eigenfunction on $\Omega$. We set $\gamma_n$ to be the limit
$$\gamma_n:=\frac{1}{\Lambda\omega_d}\lim_{r\to0}\phi(U,x_n,r).$$ 
Thus by Lemma \ref{gap42} (ii), $x_n\in\Omega^{(\gamma_n)}$. Since $\gamma_n\neq1/2$, by Lemma \ref{gap42} (iii) we have that $\gamma_n\ge 1/2+\delta$. By the monotonicity of the function $\psi_n(r):=\phi(U,x_n,r)+C_1 r$ (see Proposition \ref{mono_weiss}), we have that
$$\frac{\psi_n(r)}{\Lambda\omega_d}\ge \gamma_n\ge \frac12+\delta\ ,\quad\text{for every}\quad r>0.$$
On the other hand, fixing $r>0$, the function $x\mapsto \phi(U_0,x,r)$ is continuous and so 
$$\frac1{\Lambda\omega_d}\big(\phi(U,x_0,r)+C_1r\big)=\frac1{\Lambda\omega_d}\lim_{n\to\infty}\Big\{\phi(U,x_n,r)+C_1r\Big\}\ge \frac12+\delta.$$
Passing to the limit as $r\to0$ we obtain
$$\lim_{r\to0}\frac{\phi(U,x_0,r)}{\Lambda \omega_d}\ge \frac12+\delta,$$
which is in contradiction with the assumption $x_0\in\Omega^{(1/2)}$.
\end{proof}

\subsection{The regular part of the free boundary is Reifenberg flat}\label{sec:reifenberg}
In this section we prove the Reifenberg flatness of the regular set $Reg(\partial\Omega_U)$ defined in the previous subsection. We recall the definition of Reifenberg flatness below. For more details on the properties and the structure of the Reifenberg flat domains we refer to~\cite{kt1} and \cite{rtt}.

\begin{deff}[Reifenberg flat domains]
Let $\Omega\subset\R^d$ be an open set and let $0<\delta<1/2$, $R>0$. We say that $\Omega$ is a $(\delta,R)$-Reifenberg flat domain if:
\begin{enumerate}
\item For every $x\in\partial\Omega$ and every $0<r\le R$ there is a hyperplane $H=H_{x,r}$ containing $x$ such that 
$$\text{dist}_{\mathcal H}(B_r(x)\cap H,B_r(x)\cap\partial\Omega)<r \delta. $$
\item For every $x\in\partial\Omega$, one of the connected components of the open set $B_R(x)\cap\{x\ :\ \text{dist}(x,H_{x,R})>2\delta R\}$ is contained in $\Omega$, while the other one is contained in $\R^d\setminus\overline\Omega$.
\end{enumerate}
\end{deff}

\begin{oss}
We want to highlight here a difference between our approach and the one of Caffarelli, Shahgholian and Yeressian~\cite{csy}. In \cite[Theorem 5]{csy} it was proved that the entire positivity set $\{|U|>0\}$ is an $NTA$ domain (see Definition \ref{NTA}), which is a stronger result that can be obtained by applying the approach of \cite{agucafspr} to the first eigenfunction which in our case is strictly positive, Lipschitz continuous and non-degenerate. On the other hand this result is actually used only at the regular part of the free boundary, where it is a consequence of the Reifenberg flatness (see Theorem \ref{reifimplnta}).
\end{oss}

\begin{prop}\label{reiflatprop}
Suppose that $\Omega$ is a solution of \eqref{introP} and let $x_0\in  Reg(\partial\Omega)=\Omega^{(1/2)}$. Then $\Omega$ is Reifenberg flat in a neighborhood of $x_0$. 
\end{prop}
\begin{proof}
Fix $\delta>0$ to be chosen later. Suppose that $\Omega$ is not $(\delta,R)$-Reifenberg flat for any $R>0$. Then there are sequences $x_n\to x_0$ and $r_n\to0$ such that $\Omega$ is not $(\delta,r_n)$ flat in $B_{r_n}(x_n)$. Consider the blow-up sequence 
$$U_n(x):=U_{x_n,r_n}(x)=\frac{1}{r_n}U(x_n+xr_n).$$
By Proposition \ref{blowupexistence} and Lemma \ref{Uinftymin} we may assume that $U_n$ converges uniformly in $B_1$ to a function $U_0:\R^d\to\R^k$ which is a non-trivial local minimizer for $\mathcal F_0$. 
Let $\phi_n(r):=\phi(U_n,0,r)$ be the Weiss functional relative to $U_n$ defined in \eqref{phin}. Then we have :
\begin{itemize}
\item $\phi_n(r)=\phi(U,x_n,rr_n)$ and $\phi_n'(r)\ge -C_1r_n$, where $C_1$ is the constant from Proposition \ref{mono_weiss} ;
\item the limit $\ds\lim_{r\to 0}\phi_n(r)$ exists (see Proposition \ref{mono_weiss}) and by Lemma \ref{gap42} (ii) we have that
$$\frac1{\Lambda\omega_d}\lim_{r\to 0}\phi_n(r)=\lim_{r\to 0}\frac{|\Omega\cap B_{rr_n}(x_n)|}{|B_{rr_n}|}=\frac12\ ;$$
\item the limit $\ds\lim_{n\to \infty}\phi_n(r)$ exists and is given by the function $\phi_0(r):=\phi(U_0,0,r)$
which, for every $r_2>r_1>0$, satisfies (see Proposition \ref{mono_weiss_global})
\begin{equation}\label{monoXinfty44}
\phi_0(r_2)-\phi_0(r_1)=\int_{r_1}^{r_2}\frac{1}{r^{d+2}}\int_{\partial B_r} \left|x\cdot\nabla U_0-U_0\right|^2\,d\HH^{d-1}(x)\,dr.
\end{equation}
\end{itemize}
\emph{Step 1.} We claim that 
\begin{equation*}
\phi_0(r)=\frac{\Lambda\omega_d}2\quad\text{for every}\quad r>0.
\end{equation*}
We define $\psi_n(r)=\phi_n(r)+C_1r_nr=\phi(U,x_n,rr_n)+C_1r_nr$. In particular $\psi_n(r)$ is a non-decreasing function in $r$ such that $\ds\lim_{r\to0}\psi_n(r)=\frac{\Lambda\omega_d}2$.
We fix $\eps>0$ and let $R>0$ be such that $\ds\phi(U,x_0,R)+C_1R\le \frac{\Lambda\omega_d}2+\eps$ (such an $R$ exists since $\ds\lim_{r\to0}\phi(U,x_0,r)=\frac{\Lambda\omega_d}2$). Since 
$$\lim_{n\to\infty}\phi(U,x_n,R)=\phi(U,x_0,R),$$
and the function $r\mapsto \phi(U,x_n,r)+C_1r$ is non-decreasing, we have that for $n$ large enough
$$\frac{\Lambda\omega_d}{2}\le \phi(U,x_n,R)+C_1R\le\frac{\Lambda\omega_d}{2}+\eps.$$
Let $n$ be large enough such that $rr_n\le R$. Then we have that
$$\psi_n(r)=\phi(U,x_n,rr_n)+C_1rr_n\le \phi(U,x_n,R)+C_1R\le\frac{\Lambda\omega_d}{2}+\eps,$$
which proves that 
$$\lim_{n\to\infty}\psi_n(r)=\frac{\Lambda\omega_d}{2},$$
and, in particular, for every $r>0$ we have
$$\phi_0(r)=\lim_{n\to\infty}\phi_n(r)=\lim_{n\to\infty}\psi_n(r)=\frac{\Lambda\omega_d}{2},$$
which concludes the proof of Step 1. 

\emph{Step 2.} We now prove that, up to a rotation, $\{|U_0|>0\}=\{x_d>0\}$. We first notice that,  by \eqref{monoXinfty44}, $U_0$ is one-homogeneous. On the other hand $U_0$ is harmonic on $\Omega_0$ which gives that 
$$\frac12=\frac1{\Lambda\omega_d}\lim_{r\to0}\phi_0(r)=\lim_{r\to 0}\frac{|\Omega_0\cap B_r|}{|B_r|}.$$ 
Thus after a rotation of the coordinate axes necessarily $U_0(x)=\xi x_d^+$, for some vector $\xi\in\R^k$, which is non-zero due to Proposition \ref{blowupexistence}. In particular, we get that $\{|U_0|>0\}=\{x_d>0\}$. 

We now get the conclusion since, by Proposition~\ref{blowupexistence}, $\partial\Omega_n$ converges Hausdorff to $\{x_d=0\}$ and thus, for $n$ large enough, $\Omega_n$ is $(\delta,1)$ flat in the ball $B_1$, which is a contradiction with the initial assumption. 
\end{proof}

\subsection{The regular part of the free boundary is $C^{\infty}$}
In this last section we are finally in a position to prove our main result, Theorem~\ref{main}. For sake of simplicity we present the results in several steps, highlighting all the key points of our strategy.
First of all, in order to prove $C^{1,\alpha}$ regularity for the regular part of the boundary, we need first to introduce the notion of NTA, i.e.~\emph{non-tangentially accessible}, domains.
NTA domains were first introduced by Jerison and Kenig in the seminal paper~\cite{jk} in order to extend the boundary Harnack principle under minimal geometrical conditions, while Kenig and Toro~\cite{kt1} proved that a $(\delta,R)$-Reifenberg flat set (with $\delta$ sufficiently small) is also NTA. Roughly speaking, an NTA domain is such that every boundary point is accessible from inside and outside the domain by means of non-tangential balls. For sake of completeness, though we will just refer to the papers~\cite{jk,kt1} for the proofs and the details, we give the formal definition of NTA domain and the statements of the main Theorems.

\begin{deff}\label{NTA}
A bounded domain $\Om\subset \R^d$ is called NTA if there exist constants  $M>0$ and $r_0>0$, called NTA constants, such that 
\begin{enumerate}
\item $\Omega$ satisfies the \emph{corkscrew condition}, that is, given $x\in \partial \Om$ and $r\in(0,r_0)$, there exists $x_0\in\Om$ such that \[
M^{-1}r<dist(x_0,\partial \Om)<|x-x_0|<r,
\]
\item $\R^d\setminus \Omega$ satisfies the corkscrew condition,
\item If $w\in \partial \Om$ and $w_1,w_2\in B(w,r_0)\cap \Om$, then there is a rectifiable curve $\gamma\colon [0,1]\rightarrow \Om$ with $\gamma(0)=w_1$ and $\gamma(1)=w_2$ such that
\begin{enumerate}[(i)]
\item $\HH^1(\gamma([0,1]))\leq M|w_1-w_2|$,
\item $\min{\{\HH^1(\gamma([0,t])),\HH^1(\gamma([t,1]))\}}\leqq M dist(\gamma(t),\partial \Om)$, for every $t\in[0,1]$.
\end{enumerate}
\end{enumerate} 
\end{deff}

\begin{teo}[Reifenberg flat implies NTA, {\cite[Theorem 3.1]{kt1}}]\label{reifimplnta}
There exists a $\delta_0>0$ such that if $\Om\subset \R^d$ is a $(\delta, R)$-Reifenberg flat 
domain for $\delta <\delta_0$, then it is NTA.
\end{teo}

It was proved in \cite{jk} that in any NTA domain $\Omega\subset\R^d$ the Boundary Harnack Principle does hold, that is, if $u$ and $v$ are positive harmonic functions in $\Omega$, vanishing on the boundary $\partial\Omega\cap B_r$, then  
\[
\frac{v}{u}\ \mbox{ is H\"older continuous on }\ \overline{\Om}\cap B_r.
\]
\noindent In our setting, there are two main differences. First of all our functions $u_i$, $i=1,\dots,k$ are not harmonic, but they  solve an eigenvalue problem\[
-\Delta u_i=\la_i u_i\quad\text{in}\quad\Omega,\qquad u_i=0\quad\text{on}\quad\partial\Omega,
\]
for some $\la_i>0$. On the other hand, we do not know whether in a neighborhood of a boundary point all the $u_i$ are positive or not; this is an information that we have only on $u_1$, thanks to the non-degeneracy properties (see Lemma~\ref{nondegu1}). The case of eigenfunctions was treated in \cite[Appendix A]{rtt}. Precisely, we have the following result.

\begin{lemma}[Boundary Harnack principle for the eigenfunctions on optimal sets]\label{ntaholder1}
Let $\Omega$ be a solution of \eqref{introP}, $U=(u_1,\dots,u_k)$ be the vector of the first $k$ eigenfunctions on $\Omega$ and $0\in \Om^{(1/2)}$. Then $\Om$ is an NTA domain in a neighborhood of $0$ and 
there exists $\beta>0$, depending only on the NTA constants, such that for all $i=2,\dots, k$ 
\[
\frac{u_i}{u_1}\ \mbox{is H\"older continuous of order }\beta \mbox{ on }\overline{\Om}\cap B_r.
\]
In particular, for every $x_0\in \Om^{(1/2)}\cap B_r$, the limit 
$$\ \ds g_i(x_0):=\lim_{\Omega\ni x\to x_0}\frac{u_i(x)}{u_1(x)},\ $$  
exists and $g_i:B_r\cap\partial\Omega\to\R$ is an $\beta$-H\"older continuous function. 
\end{lemma}
\begin{proof}
By Proposition \ref{gapprop} and Proposition \ref{reiflatprop} we have that  $\partial\Omega=\Omega^{(1/2)}$ and $\Omega$ is Reifenberg flat in a sufficiently small ball $B_r$. The claim follows by \cite[Lemma A.2]{rtt} and \cite[Lemma A.3]{rtt}.
\end{proof}

In the following lemma we show that the first eigenfunction on an optimal set $\Omega$ is a solution of a one-phase free boundary problem. 

\begin{lemma}\label{optu1}
Let $\Omega$ be an optimal set for \eqref{introP} and let $u_1$ be the first eigenfunction on $\Omega$. Then, for every $x_0\in Reg(\partial\Omega)$ there is a radius $r>0$, a constant $0<c_0\le 1$ and a H\"older continuous function $g:B_{r}(x_0)\cap\partial\Omega\to [c_0,1]$ such that $u_1$ is a viscosity solution to the problem
\begin{equation*}
-\Delta u_1=\la_1(\Omega) u_1\quad\mbox{in}\quad\Om\ , \qquad u_1=0\quad \mbox{on}\quad\partial \Om\ ,\qquad |\nabla u_1|=g\sqrt\Lambda \quad \mbox{on}\quad B_{r}(x_0)\cap\partial\Omega.
\end{equation*}
\end{lemma}
\begin{proof}
Let $x_0=0$ and $U=(u_1,\dots,u_k)$ be the vector of the first $k$ eigenfunctions on $\Omega$. Let $r>0$ be the radius and $g_i:B_r\cap\partial\Omega\to\R$, for $i=2,\dots,k$ be the H\"older continuous functions from Lemma \ref{ntaholder1}. Then we have 
$$u_i=g_i u_1\quad\text{on}\quad B_r\cap\overline\Omega\qquad\text{and}\qquad u_1=g|U|\quad\text{on}\quad B_r\cap\overline\Omega,$$
where we have set 
$$g:=\frac1{\sqrt{1+g_2^2+\dots+g_k^2}}.$$
We notice that $g$ is a $\beta$-H\"older continuous function on $\overline\Omega\cap B_r$ for some $\beta>0$ and is such that $c_0\le g\le 1$, where $c_0=1/C$ and $C$ is the constant from Lemma \ref{nondegu1}. Suppose now that the function $\varphi\in C^1(\R^d)$ is touching $u_1$ from below (see Definition \ref{viscoptimality}) in a point $x_0\in\partial\Omega\cap B_r$. For $\rho$ small enough, there is a constant $C>0$ such that  
$$\frac1{g(x)}\ge \frac1{g(x_0)}-C|x-x_0|^\gamma\ge 0\quad\text{for every}\quad x\in \overline\Omega\cap B_{\rho}(x_0),$$
and so, setting $\psi(x)=\varphi(x)\big(\frac1{g(x_0)}-C|x-x_0|^\gamma\big)$, we get that $\psi(x_0)=|U|(x_0)$ and 
$$\psi(x)\le u_1(x)\left(\frac1{g(x_0)}-C|x-x_0|^\gamma\right)\le |U|(x)\quad\text{for every}\quad x\in \overline\Omega\cap B_{\rho}(x_0),$$
that is in the ball $B_{\rho}(x_0)$ we have that $\psi$ touches $|U|$ from below in $x_0$. On the other hand, $\psi$ is differentiable in $x_0$ and $|\nabla \psi(x_0)|=\frac1{g(x_0)}|\nabla \varphi(x_0)|$. Since $U$ is a viscosity solution of \eqref{viscsolUtion} we get that 
$$\sqrt\Lambda\ge |\nabla \psi(x_0)|=\frac1{g(x_0)}|\nabla \varphi(x_0)|,$$
which gives the claim, the case when $\varphi$ touches $u_1$ from below being analogous.
\end{proof}

Now the regularity of $Reg(\partial\Omega)$ follows by the already known results on the regularity of the one-phase free boundaries (see \cite{desilva} and the references therein).
\begin{prop}\label{regularityfinal}
Let $\Omega$ be a solution of \eqref{introP}. Then $Reg(\partial\Omega)=\Omega^{(1/2)}$ is locally a graph of a $C^{1,\alpha}$ function.
\end{prop}
\begin{proof}
In view of Lemma \ref{optu1} the claim follows by \cite[Theorem 1.1]{desilva}. 
\end{proof}

In order to pass from $C^{1,\alpha}$ to $C^{\infty}$ we need an improved boundary Harnack principle, as it was proved by De Silva and Savin~\cite{dss} for harmonic functions. The extension to eigenfunctions can be done as in~\cite[Appendix A]{rtt}.

\begin{lemma}[Improved boundary Harnack principle]\label{imprbh}
Let $\Omega$ be a solution of \eqref{introP}, $U=(u_1,\dots,u_k)$ be the vector of the first $k$ eigenfunctions on $\Omega$ and $0\in Reg(\partial \Om)$.
There exists $R_0<1/2$ such that, if for $r<R_0$, $Reg(\partial \Om)\cap B_r$ is of class $C^{k,\alpha}$ for $k\geq 1$, then for all $i=2,\dots, k$ we have 
\[
\frac{u_i}{u_1}\ \mbox{is of class }C^{k,\alpha} \mbox{ on }\overline{\Om}\cap B_r.
\] 
In particular, for every $x_0\in Reg(\partial \Om)\cap B_r$, the limit 
$$\ \ds g_i(x_0):=\lim_{\Omega\ni x\to x_0}\frac{u_i(x)}{u_1(x)},\ $$  
exists and $g_i:B_r\cap\partial\Omega\to\R$ is a $C^{k,\alpha}$ function. 
\end{lemma}
\begin{proof}
In order to get the claim, it is enough to apply~\cite[Theorem~2.4]{dss} for the case $k=1$ and~\cite[Theorem~3.1]{dss} for the case $k\geq 2$ to the functions $u=u_1/\varphi_0$ and $v=u_i/\varphi_0$, for all $i=2,\dots, k$, for a suitable $\varphi_0$ chosen following the ideas of~\cite[Lemma~A.2]{rtt}.
More precisely, we take $R_0>0$ such that there exists $\varphi_0\geq0$ a nontrivial function satisfying \[
-\Delta \varphi_0=\la_1(\Om)\varphi_0,\quad \mbox{in }B_{3R_0},\qquad \varphi_0=0\;\mbox{on }\partial B_{3R_0}.
\] 
Then $\varphi_0>0$ in $\overline B_{2R_0}$ and we have that $u_1/\varphi_0$ and $u_i/\varphi_0$ solve the equation \[
\div \left(\varphi_0^2\nabla (\frac{u_1}{\varphi_0})\right)=0,\qquad \div \left(\varphi_0^2\nabla (\frac{u_i}{\varphi_0})\right)=(\la_i(\Om)-\la_1(\Om))u_i\varphi_0\qquad \mbox{ in } B_{2R_0}\cap Reg(\partial \Om).
\]
\end{proof}

At this point we are in position to prove the full regularity of $Reg(\partial \Om)$.

\begin{prop}\label{regularityfinalinfty}
Let $\Omega$ be a solution of \eqref{introP}. Then $Reg(\partial\Omega)=\Omega^{(1/2)}$ is locally a graph of a $C^{\infty}$ function.
\end{prop}
\begin{proof}
The smoothness of the free boundary follows by a bootstrap argument as in \cite{kn}. 
Let us assume that $Reg(\partial \Om)$ is locally $C^{k,\alpha}$ regular for some $k\geq 1$, the case $k=1$ being true thanks to Proposition~\ref{regularityfinal}. We will prove that $Reg(\partial \Om)$ is locally $C^{k+1,\alpha}$. By Lemma~\ref{optu1} the first eigenfunction $u_1$ is a solution to the problem
\begin{equation*}
-\Delta u_1=\la_1(\Omega) u_1\quad\mbox{in}\quad\Om\ , \qquad u_1=0\quad \mbox{on}\quad  Reg(\partial \Om)\ ,\qquad |\nabla u_1|=g\sqrt\Lambda \quad \mbox{on}\quad Reg(\partial \Om).
\end{equation*}
Now thanks to Lemma~\ref{imprbh} and the definition of $g$ we have that $g$ is a $C^{k,\alpha}$ function. Now by \cite[Theorem 2]{kn} we have that $Reg(\partial\Omega)$ is locally a graph of a $C^{k+1,\alpha}$ function, and this concludes the proof.
\end{proof}

\subsection{Dimension of the singular set}\label{subs:last} In this last subsection we discuss the dimension of the singular set $Sing(\partial\Omega)=\partial\Omega\setminus Reg(\partial\Omega)$. We first notice that $\HH^{d-1}(Sing(\partial\Omega))=0$.
\begin{oss}[The singular set has $\HH^{d-1}$-measure zero]\label{remfederer}
We recall that, if $\Omega$ is a solution of \eqref{introP}, then the De Giorgi perimeter of $\Omega$ is finite, $P(\Omega)<+\infty$. In particular, by the Federer's Theorem (see, for example,~\cite[Theorem~3.61]{afp}) we obtain
\begin{equation}\label{federer}
\HH^{d-1}\big(\R^{d}\setminus(\Omega^{(1)}\cup\Omega^{(0)}\cup\Omega^{(1/2)})\big)=0.
\end{equation}
On the other hand, by the density estimate Lemma \ref{densestlemma}, we have that 
$$\partial\Omega=\R^d\setminus(\Omega^{(1)}\cup\Omega^{(0)}),$$
which together with \eqref{federer} gives 
\begin{equation*}
\HH^{d-1}\big(Sing(\partial\Omega)\big)=\HH^{d-1}\big(\partial\Omega\setminus Reg(\partial\Omega)\big)=\HH^{d-1}\big(\partial\Omega\setminus\Omega^{(1/2)}\big)=0.
\end{equation*}
\end{oss}
\noindent The above result concerning the ``smallness'' of the singular set can be improved in the following form. 

\begin{prop}\label{regbdryweiss}
Let $\Omega$ be a solution of~\eqref{introP}. There exists a critical dimension $d^*\in[5,7]$ such that $\Omega$ has the following property:
\begin{enumerate}[(a)]
\item If $d<d^*$, then $Sing(\partial \Om)$ is empty,
\item If $d=d^*$, then the singular set $Sing(\partial \Om)$ contains at most a finite number of isolated points,
\item If $d>d^*$, then the Hausdorff dimension of $Sing(\partial \Om)$ is less than $d-d^*$, that is, for every $s>0$ we have that $\HH^{d-d^\ast+s}(Sing(\partial\Omega))=0$.
\end{enumerate} 
\end{prop}
We recall that $d^*$ is the lowest dimension at which the free boundaries $\partial\{u>0\}$ of the (one-homogeneous) local minimizers $u$ of the functional 
$$H^1_{loc}(\R^d)\ni u\mapsto \mathcal E_0(u)=\int|\nabla u|^2\,dx+|\{u>0\}|,$$
admit singularities. This is related but slightly different from the case of minimal surfaces, since in our situation we have more information than the minimality with respect to the area.
Moreover, while in the theory of minimal surfaces it is well-known that the critical dimension is precisely $8$ (thanks to the works of Simons~\cite{s} and Bombieri, De Giorgi, Giusti~\cite{bdgg}), up to our knowledge (see, for example, \cite{dsj} and the recent \cite{js}) it is only known that $d^*\in[5,7]$. A reasonable conjecture, suggested by the techniques used in~\cite{cjk}, is that $d^*=7$.  

The kind of stratification result above is nowadays rather standard in the theory of minimal surfaces and it can be proved in many ways, for example by applying the well-known Federer's reduction principle (see, for example~\cite[Appendix A]{Simon}).  On the other hand, we will follow the approach of Weiss~\cite[Section 4]{weiss99}, which comes directly from the book of Giusti~\cite{giusti}.
The rest of the section is dedicated to the proof of Proposition~\ref{regbdryweiss}.

\begin{proof}[\bf Proof of Proposition \ref{regbdryweiss} (a)]
Let $U=(u_1,\dots, u_k)$ be the vector of the first $k$ eigenfunctions on $\Omega$.  Let $x_0\in \partial\Omega$ and $U_0\in \mathcal{BU}_{U}(x_0)$. By Proposition \ref{blowup} we have that $|U_0|$ is a local minimizer of the scalar Alt-Caffarelli functional. Since $d<d^\ast$, we have that $0$ is a regular point for $\partial\{|U_0|>0\}$, and in particular it has density $1/2$. Thus $\Omega$ also has density $1/2$ in $x_0$, that is 
$$\lim_{r\to 0}\frac{|\Omega\cap B_r(x_0)|}{|B_r|}=\lim_{r\to 0}\frac{|\{|U_0|>0\}\cap B_r(x_0)|}{|B_r|}=\frac12,$$
which  finally gives that $x_0\in Reg(\partial\Omega)$. Since $x_0$ is an arbitrary point of the free boundary, we obtain that $\partial\Omega=Reg(\partial\Omega)$ and $Sing(\partial\Omega)=\emptyset$.
\end{proof}

For the proof of (b) and (c) we will need some preliminary results. 

\begin{lemma}\label{flatnessimpliesregularity}
Suppose that $U\in H^1(\R^d;\R^k)$ is a Lipschitz continuous function, satisfying the quasi-minimality condition~\eqref{Uoptcond}. There are constants $\delta_0$ and $r_0$ such that :
\begin{equation*}
\text{If }\  x_0\in\partial\Omega_U\ \text{ and }\ r\le r_0\ \text{ are such that }\ \phi(U,x_0,r)\le \frac12+\delta_0\ ,\ \text{ then }\ x_0\in Reg(\partial\Omega_U),
\end{equation*}
where $\Omega_U=\{|U|>0\}$, $Reg(\partial\Omega_U)=\Omega_U^{(1/2)}$ and $\phi(U,x_0,r)$ is the Weiss functional from \eqref{weiss_fun}.
\end{lemma}
\begin{proof}
Suppose that $x_0\in\partial\Omega_U$ is such that $\ds\phi(U,x_0,r)\le \frac12+\delta_0$ and let $C_1$ be the constant from Proposition \ref{mono_weiss}. Then the function $r\mapsto \phi(U,x_0,r)+C_1 r$ is non-decreasing and so, taking into account the fact that the density is the limit of the Weiss functional  \eqref{density=phi}, we obtain 
$$\lim_{r\to 0}\frac{|\Omega\cap B_r(x_0)|}{|B_r|}=\lim_{r\to0}\phi(U,x_0,r)\le \frac12+\delta_0+C_1 r_0.$$
Choosing, $\delta_0$ and $r_0$ such that $\delta_0+C_1 r_0\le \gamma$ where $\gamma$ is the constant from Lemma \ref{gap}, we get the claim by Lemma \ref{gap42}.
\end{proof}

\begin{proof}[\bf Proof of Proposition \ref{regbdryweiss} (b)] We argue as in \cite[Theorem 4.1]{weiss99}. Suppose that there are infinite points in $Sing(\partial\Omega)$. Then there is a sequence $x_n\in Sing(\partial\Omega)$ such that:
$$x_n\to x_0\in Sing(\partial\Omega)\ , \qquad r_n:=|x_n-x_0|\to 0\ ,\qquad U_{n}(x):=\frac{U(x_n+r_n x)}{r_n}\to U_0(x)\in \mathcal{BU}_U(x_0).$$
We set $\Omega_0=\{|U_0|>0\}$ and we consider two cases: \\

Case 1 : $Sing(\Omega_0)\setminus\{0\}\neq\emptyset$. Then there is a point $\xi_0\in Sing(\Omega_0)\setminus0$ and by the one-homogeneity of $u_0:=|U_0|$ we have that every point of the form $t \xi_0$, for $t>0$, is a singular point for $\Omega_0$. We can now apply directly \cite[Theorem 4.1]{weiss99} to obtain a contradiction. \\

Case 2 : $Sing(\Omega_0)\setminus\{0\}=\emptyset$. Let $\ds \xi_n=\frac{x_n-x_0}{r_n}\in\partial B_1$. Up to a subsequence we may suppose that $\xi_n$ converges to a point $\xi_0\in \partial B_1$. Now since $\xi_0$ is a regular point for $\Omega_0$, we can find some $r>0$ small enough such that 
$$\phi(U_0,\xi_0,r)\le \frac12+\frac{\delta_0}3,$$
where $\delta_0$ is the constant from Lemma \ref{flatnessimpliesregularity}. Since $U_0$ is the limit of the blow-up sequence $U_n$, by Proposition \ref{blowupexistence} we have that $\phi(U_n,\xi_0,r)\to \phi(U_0,\xi_0,r)$. Thus for $n$ large enough we have that 
$$\phi(U_n,\xi_0,r)\le \frac12+\frac{\delta_0}2.$$
Let us set for simplicity $L=\|\nabla U\|_{L^\infty}$, in particular, we have also that $L=\|\nabla U_n\|_{L^\infty}$, for every $n\in\N$.
We now notice that by the definition of $\phi$ we have the inequality 
\begin{align}
\phi(U_n,\xi_n,r)&=\frac1{r^d}\int_{B_r(\xi_n)}\Big[|\nabla U_n|^2+\Lambda\ind_{\Omega_n}\Big]dx -\frac1{r^{d+1}}\int_{\partial B_r(\xi_n)}|U_n|^2\,d\HH^{d-1}\nonumber\\
&\le\frac1{r^d}\int_{B_r(\xi_0)}\Big[|\nabla U_n|^2+\Lambda\ind_{\Omega_n}\Big]dx + \omega_d(L^2+\Lambda)\frac{(r+|\xi_n-\xi_0|)^d-(r-|\xi_n-\xi_0|)^d}{r^d}\nonumber\\
&\quad -\frac1{r^{d+1}}\int_{\partial B_r(\xi_0)}|U_n|^2\,d\HH^{d-1}+\frac{d\omega_d}{r}2L^2|\xi_n-\xi_0|\label{spostingthecenter}\\
&\le\phi(U_n,\xi_0,r) + \omega_d(L^2+\Lambda)d2^d\frac{|\xi_n-\xi_0|}r+\frac{d\omega_d}{r}2L^2|\xi_n-\xi_0|\nonumber\\
&=\phi(U_n,\xi_0,r) + d\omega_d(2^{d+1}L^2+\Lambda)\frac{|\xi_n-\xi_0|}r,\nonumber
\end{align}
where in the last inequality we used that, for $n$ large, $\frac{|\xi_n-\xi|}r<1/2$.
Now choosing, $n$ large enough we get that 
$$\phi(U_n,\xi_n,r)\le \frac12+{\delta_0},$$
which is impossible since $U_n$ satisfies the conditions of Lemma \ref{flatnessimpliesregularity}, but $\xi_n$ is a singular point for $U_n$ by hypothesis.
\end{proof}

In order to prove Proposition \ref{regbdryweiss} (c),  we need another preliminary result analogous to~\cite[Lemma 4.2]{weiss99}.
The main difference is that, instead of applying the epsilon regularity result~\cite[Theorem 8.2]{altcaf}, we have at our disposal Lemma~\ref{flatnessimpliesregularity} which, in fact, is an epsilon regularity result expressed in terms of the Weiss functional $\phi$.

\begin{lemma}\label{weiss99-lemma42}
Let $U\in H^1(\R^d;\R^k)$ be a Lipschitz continuous function, $\|\nabla U\|_{L^\infty}=L<+\infty$, satisfying the quasi-minimality condition \eqref{Uoptcond}. Let $\Omega_U=\{|U|>0\}$, $x_0\in \partial\Omega_U$ and $U_0\in\mathcal{BU}_U(x_0)$. Let $r_n$ be an infinitesimal sequence and $\ds U_n(x):=\frac1{r_n}U(x_n+r_n x)$ the corresponding blow-up sequence with center $x_0$ converging to $U_0$. Let $\Omega_n:=\{|U_n|>0\}$ and $\Omega_0:=\{|U_0|>0\}$. Then, for every compact set $\mathcal K\subset\R^d$ and every open set $D$ such that $Sing(\partial\Omega_0)\cap \mathcal K\subset D$, there is some $n_0>0$ such that $Sing(\partial\Omega_n)\cap \mathcal K\subset D$, for every $n\ge n_0$.
\end{lemma}
\begin{proof}
Suppose for the sake of contradiction that this is not the case. Then, there is a sequence $x_n\in Sing(\partial\Omega_n)\cap \mathcal K\setminus D$ converging to some $x_0\in Sing(\partial\Omega_0)\cap \mathcal K\subset D$. 
We notice that by the Hausdorff convergence of the free boundaries (see Proposition \ref{blowupexistence}), we have necessarily $x_0\in \partial\Omega_0$ and so $x_0\in Reg(\Omega_0)$. Thus, we can fix some $0<r<r_0$ such that 
$$\phi(U_0,x_0,r)\le \frac12+\frac{\delta_0}3,$$
where $\phi$ is the Weiss functional and $r_0,\delta_0$ are the constants from Lemma \ref{flatnessimpliesregularity}. By the convergence of $U_n$ to $U_0$ we have that for $n$ large enough
$$\phi(U_n,x_0,r)\le \frac12+\frac{\delta_0}2.$$
Now, using the estimate \eqref{spostingthecenter} for $x_n$ and $x_0$ instead of $\xi_n$ and $\xi_0$ we have that for $n$ large enough
$$\phi(U_n,x_n,r)\le \phi(U_n,x_0,r) + d\omega_d(2^{d+1}L^2+\Lambda)\frac{|\xi_n-\xi_0|}r\le \frac12+{\delta_0}.$$
Now, by Lemma \ref{flatnessimpliesregularity} we have that $x_n\in Reg(\partial\Omega_n)$ in contradiction with the initial assumption. 
\end{proof}

\begin{proof}[\bf Proof of Proposition \ref{regbdryweiss} (c)]
Suppose that for some $s>0$ we have $\HH^{d-d^\ast+s}(Sing(\partial\Omega))>0$.
By Lemma \ref{weiss99-lemma42}, \cite[Lemma 4.3]{weiss99} and \cite[Lemma 4.4]{weiss99} we have that there is some point $x_0\in \partial\Omega$ and a blow-up limit $U_0\in\mathcal{BU}_U(x_0)$ such that the set $\Omega_0=\{|U_0|>0\}$ satisfies $\HH^{d-d^\ast+s}(Sing(\partial\Omega_0))>0$. Since $|U_0|$ is a minimizer of the scalar Alt-Caffarelli function $\mathcal E_0$, this is in contradiction with the dimension of the singular set of $\partial \Omega_0$ (see \cite[Theorem 4.5]{weiss99}).
\end{proof}

\section{A free-boundary problem for vector-valued functions. Proof of Theorem~\ref{altcafvect}}\label{sec:thm1.4}
In this final Section we prove Theorem~\ref{altcafvect} following step by step the proof of Theorem~\ref{main}.
\subsection{Existence} The existence of a solution of \eqref{eqaltcafvect} follows by a standard argument in the calculus of variations; the proof is precisely the same as in the scalar case (see \cite[Theorem 1.3]{altcaf}). From now on we suppose that the vector-valued function $U=(u_1,\dots,u_k)\in H^1(\Dr;\R^k)$ is a solution of \eqref{eqaltcafvect} and we set $\Omega=\{|U|>0\}$. As in the scalar case, each component of $U$ is harmonic on $\Omega$.

\subsection{Lipschitz continuity of the minimizers} 
Let $i\in\{1,\dots,k\}$ and let $B_r\subset\Dr$ for some $r>0$. Then the optimality of $U$ implies that for every function $\tilde u_i$ such that $\tilde u_i-u_i\in H^1_0(B_r)$ we have 
$$\int |\nabla u_i|^2\,dx+\Lambda|\{u_1^2+\dots+u_i^2+\dots+u_k^2>0\}|\le \int |\nabla \tilde u_i|^2\,dx+\Lambda |\{u_1^2+\dots+\tilde u_i^2+\dots+u_k^2>0\}|, $$
which gives that
$$\int |\nabla u_i|^2\,dx\le \int |\nabla \tilde u_i|^2\,dx+\Lambda |B_r|\quad\text{for every}\quad \tilde u_i\quad \text{such that}\quad \tilde u_i-u_i\in H^1_0(B_r),$$
that is each component $u_i$ is a quasi-minimizer of the Dirichlet energy. Applying \cite[Theorem~3.3]{bmpv} we get that $u_i$, and so $U$, is Lipschitz continuous in $\Dr$. In particular, $\Omega$ is open and $u_1>0$ in $\Omega$.

\subsection{Non-degeneracy of $U$} 
We first notice that $U$ satisfies the condition \eqref{intoptcondoss} in $\Dr$ with $K=0$ and there is no restriction on the perturbations $\tilde U$, formally $\eps=+\infty$. Thus, we can apply Lemma \ref{nondegcaflemma} obtaining that there are contants $c_0>0$ and $r_0>0$, depending on $d$ and $\Lambda$ such that for every $x_0\in \Dr$ and $0<r\le \inf\{r_0,\text{dist}(x_0,\partial \Dr)\}$ the following implication holds:
\begin{equation*}
\Big(\|U\|_{L^\infty(B_{2r})}<c_0 r\Big)\Rightarrow \Big(U\equiv 0\ \ \text{in}\ \ B_r(x_0)\Big).
\end{equation*}
As in Section~\ref{sec:nondeg} it is straightforward to deduce that 
\begin{itemize}
\item $|U|$ is subharmonic, that is $\ds \Delta |U|\ge 0\quad \text{on}\quad\Dr$ (see Remark \ref{Usub});
\item $|U|\le Cu_1$ on $\Omega$, for some constant $C>0$ (see Lemma \ref{nondegu1}; notice that the fact that $\Delta U=0$ in $\Omega$ significantly simplifies the proof since this time we can take $v=|U|$ and avoid the questions involving the non-degeneracy of $|\nabla U|$);
\item there are constants $r_0>0$ and $\eps_0>0$ such that $\Omega$ satisfies the density estimate (see Lemma \ref{densestlemma})
	\begin{equation*}
	\eps_0|B_r|\le \big|\Omega\cap B_r(x_0)\big|\le (1-\eps_0)|B_r|,\quad\text{for every}\quad x_0\in \partial\Omega\cap\Dr\quad\text{and}\quad r\le r_0.   
	\end{equation*}
\end{itemize}

\subsection{Weiss monotonicity formula} The functional $\phi(U,x,r)$ defined in \eqref{weiss_fun} is monotone with respect to $r$ and satisfies the inequality 
$$\frac{d}{dr}\phi(U,x,r)\ge \frac{1}{r^{d+2}}\sum_{i=1}^k\int_{\partial B_r(x)} |x\cdot\nabla u_i-u_i|^2\,dx,$$
for every $r>0$ such that $B_r(x)\subset \Dr$ and $x\in\partial\Omega$. For the proof we refer to Proposition \ref{mono_weiss_global}.

\subsection{Structure of the blow-up limits} Setting $U_{r,x_0}(x)=\frac1rU(x_0+rx)$ we have that, up to a subsequence $r_n\to0$,  $U_{r_n,x_0}$ converges to a function $U_0:\R^d\to\R^k$ (see Proposition \ref{blowupexistence}). The structure of the blow-up limits is precisely the one described in Proposition \ref{blowup}, that is the blow-up limit $U_0$ is of the form $U_0=\xi_0|U_0|$ with $\xi\in \partial B_1\subset\R^k$ and $u=|U_0|$ being a one-homogeneous non-trivial global minimizer of the scalar Alt-Caffarelli functional $\mathcal F_0$ in the sense of Definition \ref{locmin}. The proof is precisely the same as in the case of the spectral functional (we notice that Section \ref{sec:blowup} concerns only functions satisfying the more general quasi-minimality condition~\eqref{Uoptcond}) and  is based on the Weiss' monotonicity formula and on the Lipschitz continuity and the non-degeneracy of the minimizer $U$. 

\subsection{Regularity of the free boundary} The regularity of the free boundary is based on the fact that $U$ is a viscosity solution (in sense of Definition \ref{viscoptimality} with $\lambda_1=\dots=\lambda_k=0$) to the problem 
\begin{equation*}
\Delta U=0\quad\text{in}\quad\Omega,\qquad U=0\quad\text{on}\quad\partial\Omega\cap \Dr,\qquad  |\nabla |U||=\sqrt{\Lambda}\quad\text{on}\quad\partial\Omega\cap\Dr.
\end{equation*}
The proof is precisely the one of Lemma \ref{optvisc} and is based on the structure of the blow-up limits described above. All the results in the rest of Section~\ref{sec:regularity} hold true in this setting. 
\begin{itemize}
\item Lemma \ref{gap42} holds for the solutions of \eqref{eqaltcafvect} and the density of the set $\Omega=\{|U|>0\}$ is determined by the monotone function $\phi$, that is
$$\lim_{r\rightarrow 0}\frac{|\Om\cap B_r(x_0)|}{|B_r|}=\frac{1}{\Lambda\omega_d}\lim_{r\rightarrow 0}{\phi(U,x_0,r)},\quad\text{for every}\quad x_0\in\partial\Omega\cap \Dr.$$
\item The regular part of the free boundary, defined as $Reg(\partial\Omega)=\Omega^{(1/2)}$, is an open subset of $\partial\Omega\cap\Dr$. The proof of this fact is given in Proposition \ref{gapprop} with the additional simplification due to the fact that $C_1=0$ and $\phi_n=\psi_n$.
\item The set $\Omega$ is Reifenberg flat in a neighborhood of any point $x_0\in Reg(\partial\Omega)$.  The proof is given in Proposition \ref{reiflatprop} where again we have $C_1=0$ and $\phi_n=\psi_n$. 
\item The Reifenberg flatness of $Reg(\partial\Omega)$ together with {\cite[Theorem 3.1]{kt1}} and \cite{jk} imply that the set $\Omega$ satisfies a Boundary Harnack Principle at the flat free boundary points. Now the positivity of $u_1$ and the optimality condition $|\nabla|U||=\sqrt\Lambda$ give that $u_1$ is a viscosity solution of the problem
$$\Delta u_1=0\quad\mbox{in}\quad\Om\ , \qquad u_1=0\quad \mbox{on}\quad \partial\Omega\cap\Dr ,\qquad |\nabla u_1|=g\sqrt\Lambda \quad \mbox{on}\quad Reg(\partial\Omega),$$ 
where $g:\Omega\to\R$ is a smooth function with a $C^{0,\alpha}$ extension to $Reg(\partial\Omega)$. For the proof we refer to Lemma \ref{optu1}. We notice that the optimality condition in viscosity sense can be alternatively stated as
$$\Delta u_1=0\quad\mbox{in}\quad\Om\ , \qquad u_1=0\quad \mbox{on}\quad \partial\Omega\cap\Dr ,\qquad |\nabla u_1|=g\sqrt\Lambda \quad \mbox{on}\quad \partial\Omega\cap\Dr.$$
In fact, if a smooth test function touches $u_1$ in a boundary point, then this point is necessarily part of the regular free boundary $Reg(\partial\Omega)$. 
\item Applying \cite[Theorem 1.1]{desilva} we get that $Reg(\partial\Omega)$ is locally a graph of a $C^{1,\alpha}$ function. By the improved boundary Harnack principle of De Silva and Savin~\cite{dss} for harmonic functions (see Lemma \ref{imprbh}), we get that $Reg(\partial\Omega)$ is $C^\infty$. The estimate of the dimension of the singular set $Sing(\partial\Omega)=\partial\Omega\setminus Reg(\partial\Omega)$ is classical and we refer to Subsection \ref{subs:last} for more details.
\end{itemize}

\end{document}